\documentclass[preprint,12pt]{elsarticle1}

\usepackage{subfig}
\usepackage{graphicx}
\usepackage{caption,color}   
\usepackage{amssymb}
\usepackage{amsthm}
\usepackage{amsmath}
\usepackage{epic}
\usepackage{setspace}
\usepackage{float}
\usepackage{multirow}
\newtheorem{thm}{Theorem}[section]

\newtheorem{lem}[thm]{Lemma}

\newtheorem{remark}[thm]{Remark}

\numberwithin{equation}{section}
\journal{}

\begin{document}
\begin{spacing}{1.15}
\begin{frontmatter}
\title{Lexicographical ordering by spectral moments of bicyclic hypergraphs}

\author{Hong Zhou}
\author{Changjiang Bu}\ead{buchangjiang@hrbeu.edu.cn}
\address{School of Mathematical Sciences, Harbin Engineering University, Harbin 150001, PR China}

\begin{abstract}

For bicyclic hypergraphs, ordering by spectral moment ($S$-order) is investigated in this paper.
We give the first and last hypergraphs in an $S$-order of linear bicyclic uniform hypergraphs with given girth and number of edges.

\end{abstract}

\begin{keyword}
hypergraph, spectral moment, adjacency tensor\\
\emph{MSC:}
05C65, 15A18
\end{keyword}
\end{frontmatter}

\section{Introduction}
For a $k$-uniform hypergraph $\mathcal{H}$ with $n$ vertices, the sum of $d$ powers of all eigenvalues of the adjacency tensor $\mathcal{A}_\mathcal{H}$ of $\mathcal{H}$ is called the \textit{$d$th order  spectral moment} of $\mathcal{H}$, denoted by $\mathrm{S}_{d}(\mathcal{H})$. The number of eigenvalues of $\mathcal{A}_\mathcal{H}$ is $n(k-1)^{n-1}$ \cite{qi2005eigenvalues,cooper2012spectra}.
For two $k$-uniform hypergraphs $\mathcal{H}_{1}$ and $\mathcal{H}_{2}$ with $n$ vertices,
if there exists an $l\in\{1,2,\ldots,n(k-1)^{n-1}-1\}$ such that $\mathrm{S}_{i}(\mathcal{H}_{1})=\mathrm{S}_{i}(\mathcal{H}_{2})$ for $i=0,1,\ldots,l-1$ and
$\mathrm{S}_{l}(\mathcal{H}_{1})<\mathrm{S}_{l}(\mathcal{H}_{2})$, write $\mathcal{H}_{1}\prec_{s} \mathcal{H}_{2}$.
If $\mathrm{S}_{i}(\mathcal{H}_{1})=\mathrm{S}_{i}(\mathcal{H}_{2})$ for $i=0,1,\ldots,n(k-1)^{n-1}-1$, write $\mathcal{H}_{1}=_{s} \mathcal{H}_{2}$. The above order is called the \textit{$S$-order} of hypergraphs \cite{2309.16925}. Clearly, when $\mathcal{H}_{1}$ and $\mathcal{H}_{2}$ are 2-uniform, the above order is the $S$-order of graphs.

The $S$-order of graphs had been used in producing graph catalogues \cite{Drago1984A}.
In 1987, Cvetkovi\'{c} and Rowlinson \cite{cvetkovic1987spectra} gave some definitions concerning orderings of graphs including $S$-order.
And the first and last graphs were given in an $S$-order of trees with given number of vertices, unicyclic graphs with given girth and number of vertices, respectively.
Wu and Fan \cite{2014On1111} characterized the first and
last graphs in an $S$-order of bicyclic graphs with given number of vertices.
Some results on an $S$-order have been obtained in the literature:
see \biboptions{sort&compress}\cite{WU20101707,PAN20111265,CHENG2012858} for trees, \cite{CHENG20121123} for unicyclic graphs and  \cite{SHUCHAO2013ON,10.1216/RMJ-2016-46-1-261} for connected graphs.

The study of spectral hypergraph theories via tensors has attracted extensive attention \biboptions{sort&compress}\cite{clark2021harary,cooper2,doi:10.1137/21M1404740,2022Thestabilizing}, especially the spectral moment of hypergraphs \cite{clark2021harary,ref23}.
Shao et al. established some formulas for the trace of tensors in terms of some graph parameters \cite{shao2015some}.
When the tensor is an adjacency tensor, the above formulas are the formulas for spectral moments of hypergraphs.
Fan et al. gave a formula for the spectral moment of
hypergraphs in terms of the number of arborescences \cite{FAN202389}.
Clark and Cooper gave an expression for the spectral moment of
hypergraphs and used this result to give the ``Harary-Sachs'' coefficient theorem for hypergraphs \cite{clark2021harary}.
A formula for the spectral moment of hypertrees and power hypergraphs was given, respectively \cite{doi:10.1080/03081087.2021.1953431,ref23}.


In this paper, for $k\geq 3$,
the explicit expressions of $2k$th and $3k$th order spectral moments of linear bicyclic $k$-uniform  hypergraphs are
obtained in terms of the number of some subhypergraphs. And we get when $k\nmid d$, the $d$th order spectral moment of linear bicyclic $k$-uniform  hypergraphs is equal to $0$.
Using expressions of spectral moments and operations of moving edges, we give the first and last hypergraphs in an $S$-order of linear bicyclic uniform hypergraphs with given girth and number of edges.

\section{Preliminaries}

A hypergraph $\mathcal{H}$ is called \textit{$k$-uniform} if every edge of $\mathcal{H}$ contains exactly $k$ vertices.
A vertex of $\mathcal{H}$ is called a \textit{cored vertex} if it has degree one.
An edge $e$ of $\mathcal{H}$ is called a \textit{pendant edge} if it contains exactly $|e|-1$ cored vertices.
A cored vertex in a pendant edge is also called a \textit{pendant vertex}.
A hypergraph $\mathcal{H}$ is called \textit{linear} if any two edges intersect into at most one vertex.
The \textit{$k$-power hypergraph} $G^{(k)}$ is the hypergraph which is obtained by adding $k-2$ vertices with degree one to each edge of the graph $G$.
Let $P_{q}$, $S_{q}$ and $C_{q}$ denote the path, the star and the cycle with $q$ edges, respectively.
A $k$-uniform hypergraph $\mathcal{H}$ is called $k$-partite, if the vertices of $\mathcal{H}$ can be partitioned into $k$ sets so that every edge uses exactly one vertex from each set \cite{cooper2012spectra}. A $k$-uniform hypergraph $\mathcal{H}=(V,E)$ is called $hm$-bipartite if there is a disjoint partition of the vertex set $V$ as $V=V_{1}\bigcup V_{2}$ such that $V_{1},V_{2}\neq\emptyset$ and every edge in $E$ intersects $V_{1}$ in exactly one vertex and $V_{2}$ the other $k-1$ vertices \cite{2014cccA}.
Let $\mathcal{H} =(V,E)$ be a $k$-uniform hypergraph.  If for every edge $e\in E$, there is a vertex
$i_{e}\in e$ such that the degree of the vertex $i_{e}$ is one, then $\mathcal{H}$ is a cored hypergraph \cite{2013Cored}.
The spectrum of a $k$-uniform hypergraph is said to be \textit{$k$-symmetric} if this spectrum is invariant under a rotation of an angle $2\pi/k$ in the complex plane \cite{shao2015some}.
A connected $k$-uniform hypergraph $\mathcal{H}$ with $n$ vertices and $m$ edges is called bicyclic if $n=m(k-1)-1$.
For $k\geq 3$, all linear bicyclic $k$-uniform hypergraphs with $n$ vertices and $m$ edges consist of the following two types $\mathcal{B}_{n}^{k}$ and $\mathcal{C}_{n}^{k}$.

For three integers $p, q, l$ with $q\geq p\geq3$ and $l\geq 0$, let $C_{1} (resp., C_{2})$ be a $k$-uniform
 hypercycle of length $p$ (resp., $q$) and $P = (u_{0},e_{1},u_{1},\ldots,e_{l},u_{l})$ be a $k$-uniform hyperpath
of length $l$. Let $v_{1,1}\in V(C_{1}), v_{2,1}\in V(C_{2})$ be two arbitrary vertices with degree 1, and
let $v_{1,2}\in V(C_{1}), v_{2,2}\in V(C_{2})$ be two arbitrary vertices with degree 2. For $n'= (p+q+l)(k-1)-1$, let $B_{1,n',p,l,q}^{k}$
be the $n'$-vertex $k$-uniform bicyclic hypergraph obtained by
identifying $v_{1,2}$ with $u_{0}$, and identifying $v_{2,2}$ with $u_{l}$, let $B_{2,n',p,l,q}^{k}$
be the $n'$-vertex $k$-uniform
bicyclic hypergraph obtained by identifying $v_{1,2}$ with $u_{0}$, and identifying $v_{2,1}$ with $u_{l}$, and
let $B_{3,n',p,l,q}^{k}$
be the $n'$-vertex $k$-uniform bicyclic hypergraph obtained by identifying $v_{1,1}$ with $u_{0}$, and identifying $v_{2,1}$ with $u_{l}$.

For $n\geq n'$ and $i\in \{1, 2, 3\}$, let $\mathcal{B}_{i,n,p,l,q}^{k}$ be the set of $n$-vertex $m$-edge $k$-uniform bicyclic hypergraphs
each of which contains $B_{i,n',p,l,q}^{k}$ as a sub-hypergraph, where $m=\frac{n+1}{k-1}$. And $B_{i,n',p,l,q}^{k}$ is called the base of $\mathcal{B}_{i,n,p,l,q}^{k}$.
Let $\mathcal{B}_{n}^{k}=\bigcup_{i=1}^{3}\{\mathcal{B}_{i,n,p,l,q}^{k}~|~q\geq p\geq3,l\geq 0\}$.
Moreover, let $B_{1,n,p,0,q}^{k}(m-p-q)$ denote the $k$-uniform bicyclic hypergraph
obtained from $B_{1,n',p,0,q}^{k}$ by attaching $m-p-q$ pendant edges at the unique vertex with
degree 4, where $m=\frac{n+1}{k-1}$.

Let $P_{p}=(u_{1},e_{1},u_{2},\ldots,e_{p},u_{p+1}), P_{q}=(v_{1}, f_{1}, v_{2},\ldots, f_{q}, v_{q+1})$ and $P_{l}=(w_{1},g_{1},\\w_{2},\ldots,g_{l},w_{l+1})$ be three $k$-uniform hyperpaths, and suppose $(p+q+l)(k-1)-1=n'$. For three integers $p, q, l$ with $p=1,1<q\leq l$ and $1<p\leq q\leq l$, let $C_{1,n',p,q,l}^{k}$ be the $n'$-vertex $k$-uniform bicyclic hypergraph obtained from $P_{p}, P_{q}$ and $P_{l}$
by identifying three vertices $u_{1}, v_{1}, w_{1},$ and identifying three vertices $u_{p+1}, v_{q+1}, w_{l+1}$.
For $q>1, 1\leq p\leq q-1\leq l$ and $q=1, 1<p\leq l$, let $C_{2,n',p,q,l}^{k}$ be the $n'$-vertex $k$-uniform bicyclic hypergraph obtained from $P_{p}, P_{q}$ and
$P_{l}$ by identifying three vertices $u_{1}, v_{1}, w_{1}$, identifying $u_{p+1}$ with $v_{q+1}$, and identifying
$w_{l+1}$ with $v$, respectively, where $v\in f_{q}\setminus \{v_{q}, v_{q+1}\}$.
For $q>2, 1\leq p\leq q-2\leq l$ and $q=2,1\leq p\leq l$ and $q=1, k>3,1<p\leq l$, let $C_{3,n',p,q,l}^{k}$ be the $n'$-vertex $k$-uniform bicyclic hypergraph obtained from $P_{p}, P_{q}$ and
$P_{l}$ by identifying $u_{1}$ with $v_{1}$, identifying $u_{p+1}$ with $v_{q+1}$, identifying
$w_{1}$ with $v'$, and identifying $w_{l+1}$ with $v''$, respectively, where $v'\in f_{1}\setminus \{v_{1}, v_{2}\}$ and $v''\in f_{q}\setminus \{v_{q}, v_{q+1}\}$.
(in the special case $q=1$, we choose $v'\neq v''$).

For $n\geq n'$ and $i\in \{1, 2, 3\}$, let $\mathcal{C}_{i,n,p,q,l}^{k}$ be the set of $n$-vertex $m$-edge $k$-uniform bicyclic hypergraphs
each of which contains $C_{i,n',p,q,l}^{k}$ as a sub-hypergraph, where $m=\frac{n+1}{k-1}$.
And $C_{i,n',p,q,l}^{k}$ is called the base of $\mathcal{C}_{i,n,p,q,l}^{k}$.
Let $\mathcal{C}_{n}^{k}=\{\mathcal{C}_{1,n,p,q,l}^{k} ~|~ p=1,1<q\leq l \text{~or~} 1<p\leq q\leq l\} \bigcup \{\mathcal{C}_{2,n,p,q,l}^{k}~|~ q=1,1<p\leq l \text{~or~} q>1, 1\leq p\leq q-1\leq l\} \bigcup \{\mathcal{C}_{3,n,p,q,l}^{k}~|~ q>2, 1\leq p\leq q-2\leq l\text{~or~} q=2,1\leq p\leq l  \text {~or~}q=1, k>3,1<p\leq l\}$.
Moreover, for $i\in \{1, 2\}$, let $C_{i,n,p,q,l}^{k}(m-p-q-l)$  denote the $k$-uniform bicyclic hypergraph
obtained from $C_{i,n',p,q,l}^{k}$ by attaching $m-p-q-l$ pendant edges at the vertex with
degree 3, where $m=\frac{n+1}{k-1}$.


For a positive integer $n$, let $[n]=\{1,2,\ldots,n\}$.
An order $k$ dimension $n$ complex \textit{tensor}
$
\mathcal{T}=\left( {t_{i_{1}\cdots i_{k}} } \right)
$
is a multidimensional array with $n^k$ entries on complex number field $\mathbb{C}$, where $i_{1},\ldots,i_{k}\in [n]$.
Denote the set of dimension $n$ complex vectors and the set of order $k$ dimension $n$ complex tensors by $\mathbb{C}^{n}$ and $\mathbb{C}^{[k,n]}$, respectively.
For $x=\left({x_1 ,\ldots ,x_n}\right)^\mathrm{T}\in\mathbb{C}^n$, $\mathcal{T}x^{k-1}$ is a vector in $\mathbb{C}^n$ whose $i$th component is
\begin{align*}
(\mathcal{T}x^{k-1})_i=\sum\limits_{i_{2},\ldots,i_{k}=1}^{n}t_{ii_{2}\cdots i_{k}}x_{i_{2}}\cdots x_{i_{k}}.
\end{align*}
If there exist $\lambda\in\mathbb{C}$ and a nonzero vector $x=\left({x_1 ,\ldots ,x_n}\right)^\mathrm{T}\in\mathbb{C}^n$ such that
$$\mathcal{T}x^{k-1}=\lambda x^{[k-1]},$$
then $\lambda$ is called an \textit{eigenvalue} of $\mathcal{T}$ and $x$ is an \textit{eigenvector} of $\mathcal{T}$ corresponding to $\lambda$ \cite{qi2005eigenvalues,lim2005singular},
where $x^{\left[ {k - 1} \right]}  = \left( {x_1^{k - 1} ,\ldots,x_n^{k - 1} } \right)^\mathrm{T}$.



For a $k$-uniform hypergraph $\mathcal{H}$ with $n$ vertices, its  \textit{adjacency tensor} is the order $k$ dimension $n$ tensor
$\mathcal{A}_\mathcal{H}=(a_{i_{1}i_{2}\cdots i_{k}})$ \cite{cooper2012spectra}, where
\begin{equation*}
a_{i_{1}i_{2}\cdots i_{k}}=\begin{cases}
\frac{1}{(k-1)!},& \text{if } \{i_{1},i_{2},\ldots,i_{k}\}\in E(\mathcal{H}),\notag \\
0,& \text{otherwise}.
\end{cases}
\end{equation*}

Let $\mathcal{H}$ be a $k$-uniform hypergraph with $n$ vertices. Given an ordering of the vertices of $\mathcal{H}$, let
$\mathcal{F}_{d}(\mathcal{H}):= \{{(e_{1}(v_{1}),\ldots, e_{d}(v_{d})) : e_{i}\in E(\mathcal{H}), v_{1}\leq\cdots\leq v_{d}}\}$,
be the set of $d$-tuples of ordered rooted edges, where $e_{i}(v_{i})$ is an edge $e_{i}$ with root $v_{i}\in e_{i}$ for each $i\in[d]$.
Define a rooted directed star $S_{e_{i}}(v_{i})=(e_{i}, \{(v_{i}, u) : u\in e_{i}\setminus\{v_{i}\}\})$
for each $i\in[d]$, and multi-directed graph $R(F)=\bigcup_{i=1}^{d}S_{e_{i}}(v_{i})$
associated with
$F\in\mathcal{F}_{d}(\mathcal{H})$. Let  $\mathcal{F}_{d}^{\epsilon}(\mathcal{H}):=\{F\in\mathcal{F}_{d}(\mathcal{H}): R(F) \text{~is Eulerian}\}$.
For an $F\in\mathcal{F}_{d}^{\epsilon}(\mathcal{H})$, let $V(F):=V(R(F))$, $r_{v}(F)$ be the number of edges in $F$ with $v$ as the root, and $d_{v}^{+}(F)=(k-1)r_{v}(F)$ (namely, the outdegree of $v$ in $R(F)$).
Let $\tau(F)$ be the number of arborescences of $R(F)$.
The $d$th order  spectral moment of $\mathcal{H}$ can be expressed as follows \cite{FAN202389}:
\begin{equation}\label{fan111}
\mathrm{S}_d(\mathcal{H})=d(k-1)^{n}\sum\limits_{F\in\mathcal{F}_{d}^{\epsilon}(\mathcal{H})}\frac{\tau(F)}{\prod_{v\in V(F)}d_{v}^{+}(F)}.
\end{equation}

The $d$th order spectral moments of a $k$-uniform hypergraph were given for $d=0,1,2,\ldots,k$ in \cite{cooper2012spectra}.
\begin{lem}\cite{cooper2012spectra}\label{L2}
Let $\mathcal{H}$ be a $k$-uniform hypergraph with $n$ vertices and $m$ edges. Then

(1) $\mathrm{S}_0(\mathcal{H})=n(k-1)^{n-1}$;

(2) $\mathrm{S}_d(\mathcal{H})=0$ for $d=1,2,\ldots, k-1$;

(3) $\mathrm{S}_{k}(\mathcal{H})=mk^{k-1}(k-1)^{n-k}$.

\end{lem}

\begin{lem}\label{L3121}\cite{shao2015some}
Let $\mathcal{H}$ be a $k$-uniform hypergraph whose spectrum is $k$-symmetric. If $k\nmid d$, then $\mathrm{S}_{d}(\mathcal{H})=0$.
\end{lem}
For $k\nmid d$, we give $d$th order spectral moment of linear bicyclic $k$-uniform  hypergraphs.

\begin{lem}\label{sdgg}
For $k\geq 3$, let $\mathcal{H}$ be a linear bicyclic $k$-uniform hypergraph. 
When $k\nmid d$, we have $\mathrm{S}_d(\mathcal{H})=0$. 
\end{lem}
\begin{proof}
Let $\mathcal{H}_{1}$ and $\mathcal{H}_{2}$ be two vertex-disjoint $k$-partite $k$-uniform
hypergraphs with $u\in V(\mathcal{H}_{1}), v\in V(\mathcal{H}_{2})$. Obviously, the hypergraph
obtained by identifying $u$ with $v$ is a $k$-partite hypergraph.
Since $k$-uniform hypercycles and $k$-uniform hypertrees are $k$-partite hypergraphs, the hypergraph in $\mathcal{B}_{n}^{k}$ is a $k$-partite hypergraph.

If $\mathcal{H}\in \mathcal{C}_{1,n,p,q,l}^{k}, p=1,1<q\leq l \text{~or~} p>1,p\leq q\leq l$, then the base of $\mathcal{H}$ is a cored hypergraph, that is the base of $\mathcal{H}$ is an $hm$-bipartite hypergraph. Let $\mathcal{H}_{1}$ and $\mathcal{H}_{2}$  be vertex-disjoint $k$-partite $k$-uniform hypergraph and $hm$-bipartite $k$-uniform hypergraph, respectively, where $u\in V(\mathcal{H}_{1}), v\in V(\mathcal{H}_{2})$. Obviously, the hypergraph
obtained by identifying $u$ with $v$ is an $hm$-bipartite hypergraph. Therefore, $\mathcal{H}$ is an $hm$-bipartite hypergraph.

If $\mathcal{H}\in \mathcal{C}_{2,n,1,2,1}^{3}$, the base of $\mathcal{H}$ is not a $k$-partite hypergraph, but is an $hm$-bipartite hypergraph. Therefore, $\mathcal{H}$ is an $hm$-bipartite hypergraph.

Let $P_{p}=(u_{1},e_{1},u_{2},\ldots,e_{p},u_{p+1}), P_{q}=(v_{1}, f_{1}, v_{2},\ldots, f_{q}, v_{q+1})$ and $P_{l}=(w_{1},g_{1},\\w_{2},\ldots,g_{l},w_{l+1})$ be three $k$-uniform hyperpaths.
The hypercyle is obtained from $P_{p}$ and $P_{q}$ by identifying $u_{1}$ with $v_{1}$ and identifying $u_{p+1}$ with $v_{q+1}$, respectively. The bicyclic hypergraph is obtained from the hypercycle and $P_{l}$ by identifying $w_{1}$ with $v_{1}$, and identifying
$w_{l+1}$ with $v$, respectively, where $v\in f_{q}\setminus \{v_{q}, v_{q+1}\}$. The vertices of the hypercycle and  $P_{l}$ can be partitioned into $k$ sets so that every edge uses exactly one vertex from each set, respectively.

If $\mathcal{H}\in \mathcal{C}_{2,n,1,2,l}^{3},l>1$, then $v_{1}$ and $v$ is partitioned into same sets. If $w_{1}$ and $w_{l+1}$ is partitioned into same sets, then the base of $\mathcal{H}$ is a $k$-partite hypergraph. Therefore, $\mathcal{H}$ is a $k$-partite hypergraph. If $w_{1}$ and $w_{l+1}$ is partitioned into different sets, let $w_{1},w_{l+1}$ and the cored vertex of every non-pendant edges of $P_{l}$ be partitioned into same sets. Then the base of $\mathcal{H}$ is a $k$-partite hypergraph. Therefore, $\mathcal{H}$ is a $k$-partite hypergraph.

If $\mathcal{H}\in \mathcal{C}_{2,n,p,q,l}^{3},q>2,1\leq p\leq q-1\leq l$,  When $v_{1}$ and $v$ is partitioned into different sets and $w_{1}$ and $w_{l+1}$ is partitioned into different sets, the base of $\mathcal{H}$ is a $k$-partite hypergraph. Therefore, $\mathcal{H}$ is a $k$-partite hypergraph.
When $v_{1}$ and $v$ is partitioned into different sets and $w_{1}$ and $w_{l+1}$ is partitioned into same sets. Let $w_{l+1}$ be partitioned into the set of the other cored vertex of $g_{l}$ and the other cored vertex of $g_{l}$ be partitioned into the set of $w_{l+1}$. Then the base of $\mathcal{H}$ is a $k$-partite hypergraph. Therefore, $\mathcal{H}$ is a $k$-partite hypergraph.
When $v_{1}$ and $v$ is partitioned into same sets and $w_{1}$ and $w_{l+1}$ is partitioned into same sets, the base of $\mathcal{H}$ is a $k$-partite hypergraph. Therefore, $\mathcal{H}$ is a $k$-partite hypergraph.
When $v_{1}$ and $v$ is partitioned into same sets and $w_{1}$ and $w_{l+1}$ is partitioned into different sets. Let $w_{1},w_{l+1}$ and the cored vertex of every non-pendant edges of $P_{l}$ be partitioned into same sets. Then the base of $\mathcal{H}$ is a $k$-partite hypergraph. Therefore, $\mathcal{H}$ is a $k$-partite hypergraph.

If $\mathcal{H}\in \mathcal{C}_{2,n,p,1,l}^{3},1<p\leq l$, then $v_{1}$ and $v$ is partitioned into different sets. If $w_{1}$ and $w_{l+1}$ is partitioned into different sets, then the base of $\mathcal{H}$ is a $k$-partite hypergraph. Therefore, $\mathcal{H}$ is a $k$-partite hypergraph. If $w_{1}$ and $w_{l+1}$ is partitioned into same sets, let $w_{l+1}$ be partitioned into the set of the other cored vertex of $g_{l}$ and the other cored vertex of $g_{l}$ be partitioned into the set of $w_{l+1}$. Then the base of $\mathcal{H}$ is a $k$-partite hypergraph. Therefore, $\mathcal{H}$ is a $k$-partite hypergraph.

If $\mathcal{H}\in \mathcal{C}_{2,n,p,q,l}^{k},k>3, q=1,1<p\leq l \text{~or~} q>1, 1\leq p\leq q-1\leq l$, then the base of $\mathcal{H}$ is a cored hypergraph, that is the base of $\mathcal{H}$ is an $hm$-bipartite hypergraph. Therefore, $\mathcal{H}$ is an $hm$-bipartite hypergraph.

The hypercyle is obtained from $P_{p}$ and $P_{q}$ by identifying $u_{1}$ with $v_{1}$ and identifying $u_{p+1}$ with $v_{q+1}$, respectively. The bicyclic hypergraph is obtained from the hypercycle and $P_{l}$ by identifying $w_{1}$ with $v'$, and identifying
$w_{l+1}$ with $v''$, respectively, where $v'\in f_{1}\setminus \{v_{1}, v_{2}\}$ and $v''\in f_{q}\setminus \{v_{q}, v_{q+1}\}$. The vertices of the hypercycle and  $P_{l}$ can be partitioned into $k$ sets so that every edge uses exactly one vertex from each set, respectively.

If $\mathcal{H}\in \mathcal{C}_{3,n,1,2,1}^{3}$, then $v'$ and $v''$ is partitioned into different sets, and $w_{1}$ and $w_{l+1}$ is partitioned into different sets. Hence, the base of $\mathcal{H}$ is a $k$-partite hypergraph. Therefore, $\mathcal{H}$ is a $k$-partite hypergraph.

If $\mathcal{H}\in \mathcal{C}_{3,n,1,2,l}^{3},l>1$, then $v'$ and $v''$ is partitioned into different sets. Similar to the proof of $\mathcal{H}\in \mathcal{C}_{2,n,p,1,l}^{3},1<p\leq l$, the base of $\mathcal{H}$ is a $k$-partite hypergraph. Therefore, $\mathcal{H}$ is a $k$-partite hypergraph.

If $\mathcal{H}\in \mathcal{C}_{3,n,1,3,1}^{3}$, then $w_{1}$ and $w_{l+1}$ is partitioned into different sets. If $v'$ and $v''$ is partitioned into different sets, then the base of $\mathcal{H}$ is a $k$-partite hypergraph. Therefore, $\mathcal{H}$ is a $k$-partite hypergraph. If $v'$ and $v''$ are partitioned into same sets, then cored vertices of the hypercycle be partitioned into same sets. Let $v'''\in f_{q-1}\setminus \{v_{q-1}, v_{q}\}$. Let $v''$ and $v'''$  be partitioned into the set of $v_{q}$ and $v_{q}$ be partitioned into the set of $v''$. Then the base of $\mathcal{H}$ is a $k$-partite hypergraph. Therefore, $\mathcal{H}$ is a $k$-partite hypergraph.

If $\mathcal{H}\in \mathcal{C}_{3,n,1,3,l}^{3},l>1$, $\mathcal{C}_{3,n,p,q,l}^{3},q>3,1\leq p\leq q-2\leq l$ and $\mathcal{C}_{3,n,p,2,l}^{3},1<p\leq l$, similar to the proof of $\mathcal{H}\in \mathcal{C}_{2,n,p,q,l}^{3},q>2,1\leq p\leq q-1\leq l$, the base of $\mathcal{H}$ is a $k$-partite hypergraph. Therefore, $\mathcal{H}$ is a $k$-partite hypergraph.

If $\mathcal{H}\in \mathcal{C}_{3,n,p,1,l}^{4},1<p\leq l$, then $v'$ and $v''$ is partitioned into different sets. Similar to the proof of $\mathcal{H}\in \mathcal{C}_{2,n,p,1,l}^{3},1<p\leq l$, the base of $\mathcal{H}$ is a $k$-partite hypergraph. Therefore, $\mathcal{H}$ is a $k$-partite hypergraph.

If $\mathcal{H}\in \mathcal{C}_{3,n,p,q,l}^{4},q>2,1\leq p\leq q-2\leq l$, $\mathcal{C}_{3,n,p,2,l}^{4},1\leq p\leq l$ and $\mathcal{C}_{3,n,p,q,l}^{k},k>4, q>2, 1\leq p\leq q-2\leq l\text{~or~} q=2,1\leq p\leq l  \text {~or~}q=1, 1<p\leq l$, then the base of $\mathcal{H}$ is a cored hypergraph, that is the base of $\mathcal{H}$ is an $hm$-bipartite hypergraph. Therefore, $\mathcal{H}$ is an $hm$-bipartite hypergraph.

Obviously, $k$-partite hypergraphs are $hm$-bipartite.
Since the spectrum of $hm$-bipartite $k$-uniform hypergraph is $k$-symmetric \cite{2014cccA}, the spectrum of linear bicyclic $k$-uniform hypergraphs is $k$-symmetric.
By Lemma \ref{L3121}, when $k\nmid d$, the $d$th order spectral moment of linear bicyclic $k$-uniform  hypergraphs is equal to $0$.
\end{proof}


By Equation (\ref{fan111}), the following Lemmas give the expressions of $2k$th and $3k$th order spectral moments of linear bicyclic $k$-uniform  hypergraphs in terms of the number of sub-hypergraphs.
\begin{lem}\label{w3}
Let $\mathcal{U}$ be a linear bicyclic $k$-uniform  hypergraph with $k\geq3$. Then we have
$$
\mathrm{S}_{2k}(\mathcal{U})=k^{(k-1)}(k-1)^{|V(\mathcal{U})|-k}N_{\mathcal{U}}(P_{1}^{(k)})+2k^{2k-3}(k-1)^{|V(\mathcal{U})|-2k+1}N_{\mathcal{U}}(P_{2}^{(k)}).
$$
\end{lem}
\begin{proof}
By Equation (\ref{fan111}),  we consider $F\in\mathcal{F}_{d}^{\epsilon}(\mathcal{H})$.

Case 1. All elements of $F$ correspond to the same edge of $\mathcal{U}$.

We have $\tau(F)=2^{k-1}k^{k-2},\prod_{v\in V(F)}d_{v}^{+}(F)=(2(k-1))^{k}$. The total number of such $F$ is $N_{\mathcal{U}}(P_{1}^{(k)})$.

Case 2. All elements of $F$ correspond to the subhyperpath of length $2$ of $\mathcal{U}$.

We have $\tau(F)=(k^{k-2})^{2},\prod_{v\in V(F)}d_{v}^{+}(F)=2(k-1)^{2k-1}$. The total number of such $F$ is $2N_{\mathcal{U}}(P_{2}^{(k)})$.

Hence, we get
\begin{align*}
\mathrm{S}_{2k}(\mathcal{U})&=2k(k-1)^{|V(\mathcal{U})|}(\frac{2^{k-1}k^{k-2}}{(2(k-1))^{k}}N_{\mathcal{U}}(P_{1}^{(k)})+\frac{(k^{k-2})^{2}}{2(k-1)^{2k-1}}2N_{\mathcal{U}}(P_{2}^{(k)})\\
&=k^{(k-1)}(k-1)^{|V(\mathcal{U})|-k}N_{\mathcal{U}}(P_{1}^{(k)})+2k^{2k-3}(k-1)^{|V(\mathcal{U})|-2k+1}N_{\mathcal{U}}(P_{2}^{(k)}).
\end{align*}
\end{proof}
\begin{lem}\label{sp11}
For $k\geq3$, let $\mathcal{U}$ be a linear bicyclic $k$-uniform  hypergraph. Then we have
\begin{align*}
\mathrm{S}_{3k}(\mathcal{U})
&=(k-1)^{|V(\mathcal{U})|-k}k^{k-1}N_{\mathcal{U}}(P_{1}^{(k)})+6k^{2k-3}(k-1)^{|V(\mathcal{U})|+1-2k}N_{\mathcal{U}}(P_{2}^{(k)})\\
&+3k^{3k-5}(k-1)^{|V(\mathcal{U})|+2-3k}N_{\mathcal{U}}(P_{3}^{(k)})+6k^{3k-5}(k-1)^{|V(\mathcal{U})|+2-3k}N_{\mathcal{U}}(S_{3}^{(k)})\\
&+24k^{3k-6}(k-1)^{|V(\mathcal{U})|-3k+3}N_{\mathcal{U}}(C_{3}^{(k)}).
\end{align*}
\end{lem}
\begin{proof}
By Equation (\ref{fan111}),  we consider $F\in\mathcal{F}_{d}^{\epsilon}(\mathcal{H})$.

Case 1. All elements of $F$ correspond to the same edge of $\mathcal{U}$.

We have $\tau(F)=3^{k-1}k^{k-2},\prod_{v\in V(F)}d_{v}^{+}(F)=(3(k-1))^{k}$. The total number of such $F$ is $N_{\mathcal{U}}(P_{1}^{(k)})$.

Case 2. All elements of $F$ correspond to the edge of subhyperpath of length $2$ of $\mathcal{U}$.

We have $\tau(F)=(2^{k-1}k^{k-2})k^{k-2}=2^{k-1}k^{2(k-2)},\prod_{v\in V(F)}d_{v}^{+}(F)=(2(k-1))^{k-1}3(k-1)(k-1)^{k-1}=3(k-1)^{2k-1}2^{k-1}$. The total number of such $F$ is $6N_{\mathcal{U}}(P_{2}^{(k)})$.

Case 3. All elements of $F$ correspond to the edge of subhyperpath of length $3$ of $\mathcal{U}$.

We have $\tau(F)=k^{3(k-2)},\prod_{v\in V(F)}d_{v}^{+}(F)=(2(k-1))^{2}(k-1)^{2(k-1)+k-2}=4(k-1)^{3k-2}$. The total number of such $F$ is $4N_{\mathcal{U}}(P_{3}^{(k)})$.

Case 4. All elements of $F$ correspond to the edge of subhyperstar with $3$ edges of $\mathcal{U}$.

We have $\tau(F)=k^{3(k-2)},\prod_{v\in V(F)}d_{v}^{+}(F)=3(k-1)(k-1)^{3(k-1)}=3(k-1)^{3k-2}$. The total number of such $F$ is $6N_{\mathcal{U}}(S_{3}^{(k)})$.

Case 5. All elements of $F$ correspond to the edge of subhypercycle with $3$ edges of $\mathcal{U}$.

case 5.1. We have $F=\{e_{1}(1),e_{1}(1),e_{1}(2),e_{1}(3),\ldots,e_{1}(k-1),e_{2}(k),e_{2}(k),e_{2}(k+1),e_{2}(k+2),\ldots,e_{2}(2k-2),e_{3}(2k-1),e_{3}(2k-1),e_{3}(2k),\ldots,e_{3}(3k-3)\}$ or
$\{e_{3}(1),e_{3}(1),\\e_{1}(2),e_{1}(3),\ldots,e_{1}(k-1),e_{1}(k),e_{1}(k),e_{2}(k+1),e_{2}(k+2),\ldots,e_{2}(2k-2),e_{2}(2k-1),e_{2}(2k-1),e_{3}(2k),\ldots,e_{3}(3k-3)\}$.
And $\tau(F)=8k^{3k-7}$, $\prod_{v\in V(F)}d_{v}^{+}(F)=8(k-1)^{3k-3}$.

case 5.2. We have $F=\{e_{1}(1),e_{3}(1),e_{1}(2),e_{1}(3),\ldots,e_{1}(k-1),e_{1}(k),e_{2}(k),e_{2}(k+1),e_{2}(k+2),\ldots,e_{2}(2k-2),e_{2}(2k-1),e_{3}(2k-1),e_{3}(2k),\ldots,e_{3}(3k-3)\}$.
And $\tau(F)=6k^{3k-7},\prod_{v\in V(F)}d_{v}^{+}(F)=(2(k-1))^{3}(k-1)^{3(k-2)}=8(k-1)^{3k-3}$. The total number of such $F$ is $8$.

Hence, we get
\begin{align*}
\mathrm{S}_{3k}(\mathcal{U})&=3k(k-1)^{|V(\mathcal{U})|}(\frac{3^{k-1}k^{k-2}}{(3(k-1))^{k}}N_{\mathcal{U}}(P_{1}^{(k)})+\frac{2^{k-1}k^{2(k-2)}}{3(k-1)^{2k-1}2^{k-1}}6N_{\mathcal{U}}(P_{2}^{(k)})\\
&+\frac{k^{3(k-2)}}{4(k-1)^{3k-2}}4N_{\mathcal{U}}(P_{3}^{(k)})+\frac{k^{3(k-2)}}{3(k-1)^{3k-2}}6N_{\mathcal{U}}(S_{3}^{(k)})\\
&+(2\frac{8k^{3k-7}}{8(k-1)^{3k-3}}+8\frac{6k^{3k-7}}{8(k-1)^{3k-3}})N_{\mathcal{U}}(C_{3}^{(k)}))\\
&=(k-1)^{|V(\mathcal{U})|-k}k^{k-1}N_{\mathcal{U}}(P_{1}^{(k)})+6k^{2k-3}(k-1)^{|V(\mathcal{U})|+1-2k}N_{\mathcal{U}}(P_{2}^{(k)})\\
&+3k^{3k-5}(k-1)^{|V(\mathcal{U})|+2-3k}N_{\mathcal{U}}(P_{3}^{(k)})+6k^{3k-5}(k-1)^{|V(\mathcal{U})|+2-3k}N_{\mathcal{U}}(S_{3}^{(k)})\\
&+24k^{3k-6}(k-1)^{|V(\mathcal{U})|-3k+3}N_{\mathcal{U}}(C_{3}^{(k)}).
\end{align*}
\end{proof}

The sum of the squares of the degrees of all vertices of a hypergraph $\mathcal{H}$ is called the \textit{Zagreb index} of $\mathcal{H}$, denoted by $M(\mathcal{H})$ \cite{Kau2020Energies}.
The degree of a vertex $v$ of $\mathcal{H}$ is denoted by $d_{\mathcal{H}}(v)$.

\begin{remark}
We can express $N_{\mathcal{U}}(P_{2}^{(k)})$ in Lemmas \ref{w3} and \ref{sp11} as follows:
\begin{equation}\label{zsr1}
N_{\mathcal{U}}(P_{2}^{(k)})=\sum\limits_{i\in V(\mathcal{U})}{d_{\mathcal{U}}(i)\choose2}=\frac{1}{2}\sum\limits_{i\in V(\mathcal{U})}(d_{\mathcal{U}}(i))^{2}-\frac{k|E(\mathcal{U})|}{2}=\frac{1}{2}M(\mathcal{U})-\frac{k|E(\mathcal{U})|}{2}.
\end{equation}
\end{remark}

The following Lemmas gave the hypergraphs with maximum and minimum Zagreb indices among all linear bicyclic $k$-uniform hypergraphs with $n$ vertices, $m$ edges and girth $g$.
\begin{lem}\label{dghsog}\cite{sdlfhpergnk}
For $k\geq3$ and $m\geq 2g$, when $g$ is even, $C_{1,n,\frac{g}{2},\frac{g}{2},\frac{g}{2}}^{k}(m-\frac{3g}{2})$ is the hypergraph with maximum Zagreb index among all linear bicyclic $k$-uniform hypergraphs with $n$ vertices, $m$ edges and girth $g$.
When $g$ is odd, $C_{2,n,\lfloor\frac{g}{2}\rfloor,\lceil\frac{g}{2}\rceil,\lfloor\frac{g}{2}\rfloor}^{k}(m-g-\lfloor\frac{g}{2}\rfloor)$ is the hypergraph with maximum Zagreb index among all linear bicyclic $k$-uniform hypergraphs with $n$ vertices, $m$ edges and girth $g$.

\end{lem}

\begin{lem}\label{sfdgdgg}\cite{sdlfhpergnk}
For $k\geq3$,
in linear bicyclic $k$-uniform hypergraphs with $n$ vertices, $m$ edges and girth $g$, the hypergraph with maximum degree $2$ has minimum Zagreb index.
\end{lem}

\section{The $S$-order in bicyclic hypergraphs}

In this section,
we give the first and last hypergraphs in an $S$-order of linear bicyclic uniform hypergraphs with given girth and number of edges. 

The following Theorem gives the last hypergraph in an $S$-order of all linear bicyclic $k$-uniform hypergraphs with $n$ vertices, $m$ edges and girth $g$. 
\begin{thm}
For $k\geq3$ and $m\geq 2g$, when $g$ is even, $C_{1,n,\frac{g}{2},\frac{g}{2},\frac{g}{2}}^{k}(m-\frac{3g}{2})$ is the last hypergraph in an $S$-order of all linear bicyclic $k$-uniform hypergraphs with $n$ vertices, $m$ edges and girth $g$.
When $g$ is odd, $C_{2,n,\lfloor\frac{g}{2}\rfloor,\lceil\frac{g}{2}\rceil,\lfloor\frac{g}{2}\rfloor}^{k}(m-g-\lfloor\frac{g}{2}\rfloor)$ is the last hypergraph in an $S$-order of all linear bicyclic $k$-uniform hypergraphs with $n$ vertices, $m$ edges and girth $g$.

\end{thm}
\begin{proof}
By Lemmas \ref{L2} and \ref{sdgg}, we get the $d$th order spectral moments of all linear bicyclic $k$-uniform hypergraphs with $n$ vertices, $m$ edges and girth $g$ are the equal for $d=0,1,2,\ldots,2k-1$, the first significant spectral moment is the $2k$th.
For any linear bicyclic $k$-uniform hypergraph $\mathcal{H}$ with $n$ vertices, $m$ edges and girth $g$, by Lemma \ref{w3} and Equation (\ref{zsr1}),
$\mathrm{S}_{2k}(\mathcal{H})$ increases with the increase of $M(\mathcal{H})$.
By Lemma \ref{dghsog}, the theorem holds immediately.
\end{proof}

Let $Q_{t}$ be the $k$-uniform  hypertree obtained by the coalescence of $P^{(k)}_{t-1}$ at one of its non-pendant vertices of degree $1$  adjacent to the vertex of degree $2$ in a pendant edge with $P^{(k)}_{1}$ at one of its cored vertices, and
$W_{t}$ be the linear unicyclic $k$-uniform  hypergraph obtained by the coalescence of $C^{(k)}_{t-1}$ at one of its cored vertices with $P^{(k)}_{1}$ at one of its cored vertices.

Let $\mathcal{H}$ be a $k$-uniform hypergraph with $u\in V(\mathcal{H})$. A \textit{pendant path} $P$ of $\mathcal{H}$ at $u$ is itself a hyperpath with a terminal vertex $u$ such that $d_{\mathcal{H}}(u)\geq 2$ and no vertex of $V(P)\setminus\{u\}$ is adjacent to any vertex outside $P$ in $\mathcal{H}$ \cite{2023irregularity}.
Let $P$ be a pendant path at $u$ of a uniform hypergraph $\mathcal{H}$, $u\in e, e\in E(P)$, $v\in V(\mathcal{H})\setminus V(P)$.  Write $e'=(e\setminus \{u\})\bigcup\{v\}$. Let $\mathcal{H}^{'}$ be the hypergraph with $V(\mathcal{H}^{'})=V(\mathcal{H})$ and $E(\mathcal{H}^{'})=(E(\mathcal{H})\setminus\{e\})\bigcup\{e'\}$. Then $\mathcal{H}^{'}$ is obtained from $\mathcal{H}$ by moving pendant path $P$ from $u$ to $v$.

The following Theorem gives the first hypergraph in an $S$-order of all hypergraphs in $\bigcup_{i=1}^{3}\{\mathcal{B}_{i,n,g,l,q}^{k}~|~q\geq g,l\geq 0\}$.

\begin{thm}
For $k\geq3$, when $m=2g$, $B_{3,n,g,0,g}^{k}$ is the first hypergraph in an $S$-order of all  hypergraphs in $\bigcup_{i=1}^{3}\{\mathcal{B}_{i,n,g,l,q}^{k}~|~q\geq g,l\geq 0\}$.

In an $S$-order of all  hypergraphs in $\bigcup_{i=1}^{3}\{\mathcal{B}_{i,n,g,l,q}^{k}~|~q\geq g,l\geq 0\}$,
\begin{equation*}
\text{the first hypergraph is }\begin{cases}
B_{3,n,g,g-4,g}^{k},& \text{if~} m=3g-4~(g>4),\notag \\
B_{3,n,g,g-4,g+1}^{k},& \text{if~} m=3g-3~(g\geq 4),\notag \\
B_{3,n,g,t-4,t+1}^{k},& \text{if~} m=2t+g-3~(t>g, g\geq 3),\notag \\
B_{3,n,g,t-3,t+1}^{k},& \text{if~} m=2t+g-2~(t\geq g, g\geq 3).
\end{cases}
\end{equation*}

%
\end{thm}
\begin{proof}
By Lemma \ref{sfdgdgg}, we get the first hypergraph in an $S$-order of all  hypergraphs in $\bigcup_{i=1}^{3}\{\mathcal{B}_{i,n,g,l,q}^{k}~|~q\geq g,l\geq 0\}$ is in $\{\mathcal{B}_{3,n,g,l,q}^{k}~|~q\geq g,l\geq 0\}$ and its maximum degree is $2$.
Obviously, when $m=2g$, $B_{3,n,g,0,g}^{k}$ is the first hypergraph in an $S$-order of all  hypergraphs in $\bigcup_{i=1}^{3}\{\mathcal{B}_{i,n,g,l,q}^{k}~|~q\geq g,l\geq 0\}$.
Since the $d$th order spectral moments of all linear bicyclic $k$-uniform hypergraphs with maximum degree $2$ in $\{\mathcal{B}_{3,n,g,l,q}^{k}~|~q\geq g,l\geq 0\}$   are the equal for $d=0,1,2,\ldots,3k-1$, the first significant spectral moment is the $3k$th.
For any linear bicyclic $k$-uniform hypergraph $\mathcal{H}$ with maximum degree $2$ in $\mathcal{B}_{3,n,g,l,q}^{k}$, by Lemma \ref{sp11}, for fixed $g$ and $q$,
$\mathrm{S}_{3k}(\mathcal{H})$ increases with the increase of $N_{\mathcal{H}}(P_{3}^{(k)})$.

For $m-g-l-q>0$, let $\mathcal{H}_{1}$ be a hypergraph  with maximum degree 2 in $\mathcal{B}_{3,n,g,l,q}^{k}$. We move a pendant path of $\mathcal{H}_{1}$ from its terminal vertex to a pendant vertex which is not in the above pendant path. Repeating the above transformation, $\mathcal{H}_{1}$ can be changed into the $k$-uniform bicyclic hypergraph $\mathcal{H}_{2}$
obtained from $B_{3,n',g,l,q}^{k}$ by attaching a pendant path of length $m-g-l-q$ at a vertex $u$ with
degree 1, where $n'= (p+q+l)(k-1)-1$. And $N_{\mathcal{H}_{2}}(P_{3}^{(k)})\leq N_{\mathcal{H}_{1}}(P_{3}^{(k)})$.
Let $u\in e, e\in E(\mathcal{H}_{2})$. When the number of vertices of $e$ whose degree is equal to $2$ is $4$.
If $g=3$ and $q=3$, we have $N_{\mathcal{H}_{2}}(P_{3}^{(k)})=l+4+m-g-l-q+2=m$.
If $g=3$ and $q>3$, we have $N_{\mathcal{H}_{2}}(P_{3}^{(k)})=l+4+q+m-g-l-q+2=m+3$.
If $g>3$, we have $N_{\mathcal{H}_{2}}(P_{3}^{(k)})=g+l+4+q+m-g-l-q+2=m+6$.
 When the number of vertices of $e$ whose degree is equal to $2$ is $3$.
If $g=3$ and $q=3$, we have $N_{\mathcal{H}_{2}}(P_{3}^{(k)})=l+4+m-g-l-q+1=m-1$.
If $g=3$ and $q>3$, we have $N_{\mathcal{H}_{2}}(P_{3}^{(k)})=l+4+q+m-g-l-q+1=m+2$.
If $g>3$, we have $N_{\mathcal{H}_{2}}(P_{3}^{(k)})=g+l+4+q+m-g-l-q+1=m+5$.

For $m-g-l-q=0$,
when $g=3$. If $q=3$, we have $N_{B_{3,n,3,m-6,3}^{k}}(P_{3}^{(k)})=m-3-3+4=m-2$ and $N_{B_{3,n,3,m-6,3}^{k}}(C_{3}^{(k)})=2$.
If $q>3$, we have $N_{B_{3,n,3,m-3-q,q}^{k}}(P_{3}^{(k)})=m-3-q+4+q=m+1$ and $N_{B_{3,n,3,m-3-q,q}^{k}}(C_{3}^{(k)})=1$.
When $g>3$, we have $N_{B_{3,n,g,m-g-q,q}^{k}}(P_{3}^{(k)})=g+m-g-q+4+q=m+4$.

Therefore, when $g=3$ and $q\geq 3$, $N_{B_{3,n,3,m-3-q,q}^{k}}(P_{3}^{(k)})<N_{\mathcal{H}_{2}}(P_{3}^{(k)})$, that is $\mathrm{S}_{3k}(B_{3,n,3,m-3-q,q}^{k})<\mathrm{S}_{3k}(\mathcal{H}_{2})$.
For $q>3$, $\mathrm{S}_{3k}(B_{3,n,3,m-6,3}^{k})-\mathrm{S}_{3k}(B_{3,n,3,m-3-q,q}^{k})=3k^{3k-6}(k-1)^{n-3k+2}(8(k-1)-3k)>0$, that is in an $S$-order of all  hypergraphs in $\bigcup_{i=1}^{3}\{\mathcal{B}_{i,n,3,l,q}^{k}~|~l\geq0,q\geq 3\}$, the first hypergraph is in $\{B_{3,n,3,m-3-q,q}^{k}~|~q>3,m-3-q\geq 0\}$. And the $d$th order spectral moments of all hypergraphs in $\{B_{3,n,3,m-3-q,q}^{k}~|~q>3,m-3-q\geq0\}$ are the equal for each $d\in \{0,1,2,\ldots,4k-1\}$.
Obviously, when $m=7$, $B_{3,7k-8,3,0,4}^{k}$ is the first hypergraph in an $S$-order of all  hypergraphs in $\bigcup_{i=1}^{3}\{\mathcal{B}_{i,7k-8,3,l,q}^{k}~|~l\geq0,q\geq 3\}$.
When $g>3$, we have $N_{B_{3,n,g,m-g-q,q}^{k}}(P_{3}^{(k)})<N_{\mathcal{H}_{2}}(P_{3}^{(k)})$, that is $\mathrm{S}_{3k}(B_{3,n,g,m-g-q,q}^{k})<\mathrm{S}_{3k}(\mathcal{H}_{2})$. For $g>3$, in an $S$-order of all  hypergraphs in $\bigcup_{i=1}^{3}\{\mathcal{B}_{i,n,g,l,q}^{k}~|~q\geq g,l\geq 0\}$, the first hypergraph is in $\{B_{3,n,g,m-g-q,q}^{k}~|~q\geq g,m-g-q\geq 0\}$. And the $d$th order spectral moments of all hypergraphs in $\{B_{3,n,g,m-g-q,q}^{k}~|~q\geq g,m-g-q\geq 0\}$ are the equal for each $d\in \{0,1,2,\ldots,4k-1\}$.


When $g\geq 5$, suppose $\mathcal{H}\in \{B_{3,n,g,m-g-q,q}^{k}~|~q\geq g,m-g-q\geq 0\}$, if $m-g-q=0$, we have
$N_{\mathcal{H}}(P_{4}^{(k)})=g+8+q=m+8$, $N_{\mathcal{H}}(W_{4})=0$ and $N_{\mathcal{H}}(Q_{4})=2$. If $m-g-q>0$, we have $N_{\mathcal{H}}(P_{4}^{(k)})=g+l+7+q=m+7$, $N_{\mathcal{H}}(W_{4})=0$ and $N_{\mathcal{H}}(Q_{4})=2$.
When $m>2g$, by Equation (\ref{fan111}), we get the $4k$th order spectral moment of all hypergraphs in $\{B_{3,n,g,m-g-q,q}^{k}~|~q\geq g, m>g+q\}$ is the equal. And for $m-g-q>0$, $\mathrm{S}_{4k}(B_{3,n,g,m-g-q,q}^{k})<\mathrm{S}_{4k}(B_{3,n,g,0,m-g}^{k})$.
Therefore, when $g\geq 5$, in an $S$-order of all  hypergraphs in $\bigcup_{i=1}^{3}\{\mathcal{B}_{i,n,g,l,q}^{k}~|~q\geq g,l\geq 0\}$, the first hypergraph is in
$\{B_{3,n,g,m-g-q,q}^{k}~|~q\geq g, m>g+q\}$.

Let $t\geq 5$. If $t<g$,
for the $tk$th order spectral moment of hypergraphs in $\{B_{3,n,g,m-g-q,q}^{k}~|~q\geq g,m-g-q\geq t-4\}$, since the number of the subhypergraphs with $t$ edges except $P_{t}^{(k)}$ of all hypergraphs  in $\{B_{3,n,g,m-g-q,q}^{k}~|~q\geq g,m-g-q\geq t-4\}$ is equal, we only need consider the number of $P_{t}^{(k)}$.

Suppose $\mathcal{H}\in \{B_{3,n,g,m-g-q,q}^{k}~|~q\geq g,m-g-q\geq t-4\}$,\\
if $m-g-q=t-4$ and $q\geq g$, then $N_{\mathcal{H}}(P_{t}^{(k)})=g+q+2(t-4+3)+2(t-1-2)=m+3t-4$.\\
If $m-g-q>t-4$ and $q\geq g$, then $N_{\mathcal{H}}(P_{t}^{(k)})=g+q+2(t-2)+m-g-q-t+2+t-1+t-2=m+3t-5$.

Therefore, when $t<g$, $\{B_{3,n,g,m-g-q,q}^{k}~|~ q\geq g, m-g-q>t-4\}$ have minimum $tk$th order spectral moment in  $\{B_{3,n,g,m-g-q,q}^{k}~|~q\geq g,m-g-q\geq t-4\}$.

Let $g\geq 4$.
If $t=g$,
for the $gk$th order spectral moment of hypergraphs in $\{B_{3,n,g,m-g-q,q}^{k}~|~q\geq g,m-g-q\geq g-4\}$, since the number of the subhypergraphs with $g$ edges except $P_{g}^{(k)}$ and $C_{g}^{(k)}$
of all hypergraphs  in $\{B_{3,n,g,m-g-q,q}^{k}~|~q\geq g,m-g-q\geq g-4\}$
is equal, we only need consider the number of $P_{g}^{(k)}$ and $C_{g}^{(k)}$.

Suppose $\mathcal{H}\in \{B_{3,n,g,m-g-q,q}^{k}~|~q\geq g,m-g-q\geq g-4\}$,\\
if $m-g-q=g-4$ and $q=g$, then $N_{\mathcal{H}}(P_{g}^{(k)})=2(g-3)+2(3+g-4)=4g-8$, $N_{\mathcal{H}}(C_{g}^{(k)})=2$ and $m=3g-4$.\\
If $m-g-q=g-4$ and $q>g$, then $N_{\mathcal{H}}(P_{g}^{(k)})=2(g-3)+2(3+g-4)+q=4g-8+q=m+2g-4$, $N_{\mathcal{H}}(C_{g}^{(k)})=1$ and $m>3g-4$.\\
If $m-g-q>g-4$ and $q=g$, then $N_{\mathcal{H}}(P_{g}^{(k)})=2(g-2)+m-g-q-g+2+g-1+g-2=2g-5+m-q=m+g-5$, $N_{\mathcal{H}}(C_{g}^{(k)})=2$ and $m>3g-4$.\\
If $m-g-q>g-4$ and $q>g$, then $N_{\mathcal{H}}(P_{g}^{(k)})=2(g-2)+m-g-q-g+2+g-1+g-2+q=2g-5+m$, $N_{\mathcal{H}}(C_{g}^{(k)})=1$ and $m>3g-3$.


When $m=3g-3$, we consider the $gk$th order spectral moments of $B_{3,n,g,g-4,g+1}^{k}$ and $B_{3,n,g,g-3,g}^{k}$.
\begin{align*}
&\mathrm{S}_{gk}(B_{3,n,g,g-3,g}^{k})-\mathrm{S}_{gk}(B_{3,n,g,g-4,g+1}^{k})\\
&=gk(k-1)^{n}(2^{g-1}\frac{k^{gk-2g}}{2^{g-1}(k-1)^{gk-g+1}}(N_{B_{3,n,g,g-3,g}^{k}}(P_{g}^{(k)})-N_{B_{3,n,g,g-4,g+1}^{k}}(P_{g}^{(k)}))\\
&+(2^{g}\frac{2gk^{gk-2g-1}}{2^{g}(k-1)^{gk-g}}+2\frac{2^{g}k^{kg-2g-1}}{2^{g}(k-1)^{gk-g}})(N_{B_{3,n,g,g-3,g}^{k}}(C_{g}^{(k)})-N_{B_{3,n,g,g-4,g+1}^{k}}(C_{g}^{(k)})))\\
&=gk^{gk-2g}(k-1)^{n-gk+g-1}(k-2)(g+1)>0.
\end{align*}
Therefore, when $m=3g-3$, $B_{3,n,g,g-4,g+1}^{k}$  have minimum $gk$th order spectral moment in $\{B_{3,n,g,m-g-q,q}^{k}~|~q\geq g,m-g-q\geq g-4\}$.

When $m>3g-3$, we consider the $gk$th order spectral moments of $B_{3,n,g,g-4,q}^{k}$, $B_{3,n,g,m-2g,g}^{k}$ and $\{B_{3,n,g,m-g-q,q}^{k}~|~q>g, m-g-q>g-4\}$. By Equation (\ref{fan111}), we get for $m-g-q>g-4$ and $q>g$, $\mathrm{S}_{gk}(B_{3,n,g,m-g-q,q}^{k})<\mathrm{S}_{gk}(B_{3,n,g,g-4,q}^{k})$. And we have
\begin{align*}
&\mathrm{S}_{gk}(B_{3,n,g,m-2g,g}^{k})-\mathrm{S}_{gk}(B_{3,n,g,m-g-q,q}^{k})\\
&=gk(k-1)^{n}(2^{g-1}\frac{k^{gk-2g}}{2^{g-1}(k-1)^{gk-g+1}}(N_{B_{3,n,g,m-2g,g}^{k}}(P_{g}^{(k)})-N_{B_{3,n,g,m-g-q,q}^{k}}(P_{g}^{(k)}))\\
&+(2^{g}\frac{2gk^{gk-2g-1}}{2^{g}(k-1)^{gk-g}}+2\frac{2^{g}k^{kg-2g-1}}{2^{g}(k-1)^{gk-g}})(N_{B_{3,n,g,m-2g,g}^{k}}(C_{g}^{(k)})-N_{B_{3,n,g,m-g-q,q}^{k}}(C_{g}^{(k)})))\\
&=gk^{gk-2g}(k-1)^{n-gk+g-1}(k(g+2)-2(g+1))>0.
\end{align*}
Therefore, when $m>3g-3$, $\{B_{3,n,g,m-g-q,q}^{k}~|~q>g, m-g-q>g-4\}$  have minimum $gk$th order spectral moment in $\{B_{3,n,g,m-g-q,q}^{k}~|~q\geq g,m-g-q\geq g-4\}$.

Hence, when $m=3g-4$, $B_{3,n,g,g-4,g}^{k}$ is the first hypergraph in an $S$-order of all  hypergraphs in $\bigcup_{i=1}^{3}\{\mathcal{B}_{i,n,g,l,q}^{k}~|~q\geq g,l\geq 0\}$.
When $m=3g-3$, $B_{3,n,g,g-4,g+1}^{k}$ is the first hypergraph in an $S$-order of all  hypergraphs in $\bigcup_{i=1}^{3}\{\mathcal{B}_{i,n,g,l,q}^{k}~|~q\geq g,l\geq 0\}$.
When $m=3g-2$, $B_{3,n,g,g-3,g+1}^{k}$ is the first hypergraph in an $S$-order of all  hypergraphs in $\bigcup_{i=1}^{3}\{\mathcal{B}_{i,n,g,l,q}^{k}~|~q\geq g,l\geq 0\}$.

Let $g\geq 3$. If $t>g$,
for the $tk$th order spectral moment of hypergraphs in $\{B_{3,n,g,m-g-q,q}^{k}~|~q\geq t,m-g-q\geq t-4\}$, since the number of the subhypergraphs with $t$ edges except $P_{t}^{(k)}$ and $C_{t}^{(k)}$ of all hypergraphs  in $\{B_{3,n,g,m-g-q,q}^{k}~|~q\geq t,m-g-q\geq t-4\}$ is equal, we only need consider the number of $P_{t}^{(k)}$ and $C_{t}^{(k)}$.

Suppose $\mathcal{H}\in \{B_{3,n,g,m-g-q,q}^{k}~|~q\geq t,m-g-q\geq t-4\}$,\\
if $m-g-q=t-4$ and $q=t$, then $N_{\mathcal{H}}(P_{t}^{(k)})=2(g-3)+2(3+t-4)=2g+2t-8$, $N_{\mathcal{H}}(C_{t}^{(k)})=1$ and $m=2t+g-4$.\\
If $m-g-q=t-4$ and $q>t$, then $N_{\mathcal{H}}(P_{t}^{(k)})=2(g-3)+2(3+t-4)+q=2g+2t-8+q=m+g+t-4$, $N_{\mathcal{H}}(C_{t}^{(k)})=0$ and $m>2t+g-4$.\\
If $m-g-q>t-4$ and $q=t$, then $N_{\mathcal{H}}(P_{t}^{(k)})=2(t-2)+m-g-q-t+2+g-1+g-2=t+m-q+g-5=m+g-5$, $N_{\mathcal{H}}(C_{t}^{(k)})=1$ and $m>2t+g-4$.\\
If $m-g-q>t-4$ and $q>t$, then $N_{\mathcal{H}}(P_{t}^{(k)})=2(t-2)+m-g-q-t+2+g-1+g-2+q=t+m+g-5$, $N_{\mathcal{H}}(C_{t}^{(k)})=0$ and $m>2t+g-3$.


When $m=2t+g-3$, we consider the $tk$th order spectral moments of $B_{3,n,g,t-4,t+1}^{k}$ and $B_{3,n,g,t-3,t}^{k}$.
\begin{align*}
&\mathrm{S}_{tk}(B_{3,n,g,t-3,t}^{k})-\mathrm{S}_{tk}(B_{3,n,g,t-4,t+1}^{k})\\
&=tk(k-1)^{n}(2^{t-1}\frac{k^{tk-2t}}{2^{t-1}(k-1)^{tk-t+1}}(N_{B_{3,n,g,t-3,t}^{k}}(P_{t}^{(k)})-N_{B_{3,n,g,t-4,t+1}^{k}}(P_{t}^{(k)}))\\
&+(2^{t}\frac{2tk^{tk-2t-1}}{2^{t}(k-1)^{tk-t}}+2\frac{2^{t}k^{kt-2t-1}}{2^{t}(k-1)^{tk-t}})(N_{B_{3,n,g,t-3,t}^{k}}(C_{t}^{(k)})-N_{B_{3,n,g,t-4,t+1}^{k}}(C_{t}^{(k)})))\\
&=tk^{tk-2t}(k-1)^{n-tk+t-1}(k-2)(t+1)>0.
\end{align*}
Therefore, when $m=2t+g-3$, $B_{3,n,g,t-4,t+1}^{k}$  have minimum $tk$th order spectral moment in $\{B_{3,n,g,m-g-q,q}^{k}~|~q\geq t,m-g-q\geq t-4\}$.

When $m>2t+g-3$, we consider the $tk$th order spectral moments of $B_{3,n,g,m-g-t,t}^{k}$ and $\mathcal{H}\in \{B_{3,n,g,m-g-q,q}^{k}~|~q>t,m-g-q>t-4\}$.
We have
\begin{align*}
&\mathrm{S}_{tk}(B_{3,n,g,m-g-t,t}^{k})-\mathrm{S}_{tk}(\mathcal{H})\\
&=tk(k-1)^{n}(2^{t-1}\frac{k^{tk-2t}}{2^{t-1}(k-1)^{tk-t+1}}(N_{B_{3,n,g,m-g-t,t}^{k}}(P_{t}^{(k)})-N_{\mathcal{H}}(P_{t}^{(k)}))\\
&+(2^{t}\frac{2tk^{tk-2t-1}}{2^{t}(k-1)^{tk-t}}+2\frac{2^{t}k^{kt-2t-1}}{2^{t}(k-1)^{tk-t}})(N_{B_{3,n,g,m-g-t,t}^{k}}(C_{t}^{(k)})-N_{\mathcal{H}}(C_{t}^{(k)})))\\
&=tk^{tk-2t}(k-1)^{n-tk+t-1}(k(t+2)-2(t+1))>0.
\end{align*}
Therefore, when $m>2t+g-3$, $\{B_{3,n,g,m-g-q,q}^{k}~|~q>t,m-g-q>t-4\}$  have minimum $tk$th order spectral moment in $\{B_{3,n,g,m-g-q,q}^{k}~|~q\geq t,m-g-q\geq t-4\}$.

Hence, when $m=3g-2$, $B_{3,n,g,g-3,g+1}^{k}$ is the first hypergraph in an $S$-order of all  hypergraphs in $\bigcup_{i=1}^{3}\{\mathcal{B}_{i,n,g,l,q}^{k}~|~q\geq g,l\geq 0\}$.
For $t>g$, when $m=2t+g-3$, $B_{3,n,g,t-4,t+1}^{k}$ is the first hypergraph in an $S$-order of all  hypergraphs in $\bigcup_{i=1}^{3}\{\mathcal{B}_{i,n,g,l,q}^{k}~|~q\geq g,l\geq 0\}$.
When $m=2t+g-2$, $B_{3,n,g,t-3,t+1}^{k}$ is the first hypergraph in an $S$-order of all  hypergraphs in $\bigcup_{i=1}^{3}\{\mathcal{B}_{i,n,g,l,q}^{k}~|~q\geq g,l\geq 0\}$.

\end{proof}

The following Theorem gives the first hypergraph in an $S$-order of all  hypergraphs with girth $g$ in $\mathcal{C}_{n}^{k}$.

\begin{thm}
Let $k\geq3$. When $m\geq 4$, $C_{3,n,1,2,m-3}^{k}$ is the first hypergraph in an $S$-order of all hypergraphs with girth $3$ in $\mathcal{C}_{n}^{k}$.

%

For $g>3$, when $g$ is even and $m\geq \frac{3g}{2}$, $C_{3,n,g-\frac{g+2}{2},\frac{g+2}{2},m-g}^{k}$ is the first hypergraph in an $S$-order of all hypergraphs with girth $g$ in $\mathcal{C}_{n}^{k}$.
When $g$ is odd and $m\geq \frac{3g}{2}-\frac{1}{2}$, $C_{3,n,g-\lceil\frac{g+2}{2}\rceil,\lceil\frac{g+2}{2}\rceil,m-g}^{k}$ is the first hypergraph in an $S$-order of all hypergraphs with girth $g$ in $\mathcal{C}_{n}^{k}$.
\end{thm}

\begin{proof}
By Lemma \ref{sfdgdgg}, we get the first hypergraph in an $S$-order of all  hypergraphs with girth $g$  in $\mathcal{C}_{n}^{k}$ is in $\{\mathcal{C}_{3,n,p,q,l}^{k}~|~ q>2, 1\leq p\leq q-2\leq l\text{~or~} q=2,1\leq p\leq l  \text {~or~}q=1, k>3,1<p\leq l\}$ and its maximum degree is $2$.
For $m-p-l-q>0$, let $\mathcal{H}_{1}$ be the hypergraph  with girth $g$ and maximum degree 2 in $\mathcal{C}_{3,n,p,q,l}^{k}$, we move a pendant path of $\mathcal{H}_{1}$ from its terminal vertex to a pendant vertex which is not in the above pendant path. Repeating the above transformation, $\mathcal{H}_{1}$ can be changed into the $k$-uniform bicyclic hypergraph $\mathcal{H}_{2}$
obtained from $C_{3,n',p,q,l}^{k}$ by attaching a pendant path of length $m-p-l-q$ at a vertex with
degree 1. And the number of subhyperpath of length $3$ decreases.

When $g(\mathcal{H}_{2})=3$, we consider the following 4 cases.

Case 1. $p=1, q=2, l=1$. We have $N_{\mathcal{H}_{2}}(C_{3}^{(k)})=2$. When $\mathcal{H}_{2}$ is
obtained from $C_{3,n',1,2,1}^{k}$ by attaching a pendant path of length $m-4$ at the vertex with
degree $1$ of $P_{q}$. Then $N_{\mathcal{H}_{2}}(P_{3}^{(k)})=2+3+m-4-1=m$. When $\mathcal{H}_{2}$ is
obtained from $C_{3,n',1,2,1}^{k}$ by attaching a pendant path of length $m-4$ at the vertex with
degree $1$ of $P_{p}$ or $P_{l}$. Then $N_{\mathcal{H}_{2}}(P_{3}^{(k)})=2+2+m-4-1=m-1$.
When $m=4$, $N_{C_{3,n,1,2,1}^{k}}(C_{3}^{(k)})=2$ and $N_{C_{3,n,1,2,1}^{k}}(P_{3}^{(k)})=2$.

Case 2. $p=1,q=2,l>1$. We have $N_{\mathcal{H}_{2}}(C_{3}^{(k)})=1$. When $\mathcal{H}_{2}$ is
obtained from $C_{3,n',1,2,l}^{k}$ by attaching a pendant path of length $m-3-l$ at the vertex with
degree $1$ of $P_{q}$. Then $N_{\mathcal{H}_{2}}(P_{3}^{(k)})=l+2+2+m-3-l-1+3=m+3$. When $\mathcal{H}_{2}$ is
obtained from $C_{3,n',1,2,l}^{k}$ by attaching a pendant path of length $m-3-l$ at the vertex with
degree $1$ of $P_{p}$ or $P_{l}$. Then $N_{\mathcal{H}_{2}}(P_{3}^{(k)})=l+2+2+m-3-l-1+2=m+2$.
For $m-3-l=0$, we have $N_{C_{3,n,1,2,m-3}^{k}}(C_{3}^{(k)})=1$ and $N_{C_{3,n,1,2,m-3}^{k}}(P_{3}^{(k)})=2+m-3+2=m+1$.

Case 3. $p=2, q=1, l=2, k\geq4$. We have $N_{\mathcal{H}_{2}}(C_{3}^{(k)})=2$. When $\mathcal{H}_{2}$ is
obtained from $C_{3,n',2,1,2}^{k}$ by attaching a pendant path of length $m-5$ at the vertex with
degree $1$ of $P_{q}$. Then $N_{\mathcal{H}_{2}}(P_{3}^{(k)})=4+m-5-1+4=m+2$. When $\mathcal{H}_{2}$ is
obtained from $C_{3,n',2,1,2}^{k}$ by attaching a pendant path of length $m-5$ at the vertex with
degree $1$ of $P_{p}$ or $P_{l}$. Then $N_{\mathcal{H}_{2}}(P_{3}^{(k)})=4+m-5-1+2=m$. When $m=5$, we have $N_{C_{3,n,2,1,2}^{k}}(C_{3}^{(k)})=2$ and $N_{C_{3,n,2,1,2}^{k}}(P_{3}^{(k)})=4$.

Case 4. $p=2,q=1,l>2, k\geq4$. We have $N_{\mathcal{H}_{2}}(C_{3}^{(k)})=1$. When $\mathcal{H}_{2}$ is
obtained from $C_{3,n',2,1,l}^{k}$ by attaching a pendant path of length $m-3-l$ at the vertex with
degree $1$ of $P_{q}$. Then $N_{\mathcal{H}_{2}}(P_{3}^{(k)})=1+l+4+m-3-l-1+4=m+5$. When $\mathcal{H}_{2}$ is
obtained from $C_{3,n',2,1,l}^{k}$ by attaching a pendant path of length $m-3-l$ at the vertex with
degree $1$ of $P_{p}$ or $P_{l}$. Then $N_{\mathcal{H}_{2}}(P_{3}^{(k)})=1+l+4+2+m-3-l-1=m+3$.
For $m-3-l=0$, we have $N_{C_{3,n,2,1,m-3}^{k}}(C_{3}^{(k)})=1$ and $N_{C_{3,n,2,1,m-3}^{k}}(P_{3}^{(k)})=1+m-3+4=m+2$.


Therefore, when $m>4$, we compare the $3k$th order spectral moments of $\mathcal{H}_{2}$ and $C_{3,n,1,2,m-3}^{k}$, where $\mathcal{H}_{2}$ is
obtained from $C_{3,n',1,2,1}^{k}$ by attaching a pendant path of length $m-4$ at the vertex with
degree $1$ of $P_{p}$ or $P_{l}$. By Lemma \ref{sp11}, we have $S_{3k}(C_{3,n,1,2,m-3}^{k})-S_{3k}(\mathcal{H}_{2})=3k^{3k-5}(k-1)^{n+2-3k}2-24k^{3k-6}(k-1)^{n-3k+3}=6k^{3k-6}(k-1)^{n+2-3k}(-3k+4)<0.$
Hence, $C_{3,n,1,2,m-3}^{k}$ is the first hypergraph in an $S$-order of all  hypergraphs with girth $3$ in $\mathcal{C}_{n}^{k}$.
When $m=4$, obviously, $C_{3,n,1,2,1}^{k}$ is the first hypergraph in an $S$-order of all  hypergraphs with girth $3$ in $\mathcal{C}_{n}^{k}$.

Hence, when $m\geq 4$, $C_{3,n,1,2,m-3}^{k}$ is the first hypergraph in an $S$-order of all  hypergraphs with girth $3$ in $\mathcal{C}_{n}^{k}$.

When $g(\mathcal{H}_{2})>3$, we consider the following 3 cases.

Case 1. $q=1$. When $\mathcal{H}_{2}$ is
obtained from $C_{3,n',p,1,l}^{k}$ by attaching a pendant path of length $m-p-l-1$ at the vertex with
degree $1$ of $P_{q}$. Then $N_{\mathcal{H}_{2}}(P_{3}^{(k)})=p+1+l+1+4+m-p-1-l-1+4=m+8$. When $\mathcal{H}_{2}$ is
obtained from $C_{3,n',p,1,l}^{k}$ by attaching a pendant path of length $m-p-l-1$ at the vertex with
degree $1$ of $P_{p}$ or $P_{l}$. Then $N_{\mathcal{H}_{2}}(P_{3}^{(k)})=p+1+l+1+4+m-p-1-l-1+2=m+6$.
Since $N_{C_{3,n,p,1,m-p-1}^{k}}(P_{3}^{(k)})=p+1+1+m-p-1+4=m+5$, $N_{\mathcal{H}_{2}}(P_{3}^{(k)})>N_{C_{3,n,p,1,m-p-1}^{k}}(P_{3}^{(k)})$.

Case 2. $q=2$. When $\mathcal{H}_{2}$ is
obtained from $C_{3,n',p,2,l}^{k}$ by attaching a pendant path of length $m-p-l-2$ at the vertex with
degree $1$ of $P_{q}$. Then $N_{\mathcal{H}_{2}}(P_{3}^{(k)})=p+2+l+2+2+m-p-2-l-1+3=m+6$. When $\mathcal{H}_{2}$ is
obtained from $C_{3,n',p,2,l}^{k}$ by attaching a pendant path of length $m-p-l-2$ at the vertex with
degree $1$ of $P_{p}$ or $P_{l}$. Then $N_{\mathcal{H}_{2}}(P_{3}^{(k)})=p+2+l+2+2+m-p-2-l-1+2=m+5$. Since $N_{C_{3,n,p,2,m-p-2}^{k}}(P_{3}^{(k)})=p+2+2+m-p-2+2=m+4$, $N_{\mathcal{H}_{2}}(P_{3}^{(k)})>N_{C_{3,n,p,2,m-p-2}^{k}}(P_{3}^{(k)})$.

Case 3. $q>2$. When $\mathcal{H}_{2}$ is
obtained from $C_{3,n',p,q,l}^{k}$ by attaching a pendant path of length $m-p-l-q$ at the vertex with
degree $1$ of $f_{1}$ or $f_{q}$ of $P_{q}$. Then $N_{\mathcal{H}_{2}}(P_{3}^{(k)})=p+q+l+q-(q-2)+2+m-p-q-l-1+3=m+6$. When $\mathcal{H}_{2}$ is
obtained from $C_{3,n',p,q,l}^{k}$ by attaching a pendant path of length $m-p-l-q$ at the vertex with
degree $1$ of $P_{p}$, $P_{l}$ or $f_{2},\ldots, f_{q-1}$ of $P_{q}$. Then $N_{\mathcal{H}_{2}}(P_{3}^{(k)})=p+q+l+q-(q-2)+2+m-p-q-l-1+2=m+5$. Since $N_{C_{3,n,p,q,m-p-q}^{k}}(P_{3}^{(k)})=p+q+q+m-p-q+2-(q-2)=m+4$, $N_{\mathcal{H}_{2}}(P_{3}^{(k)})>N_{C_{3,n,p,q,m-p-q}^{k}}(P_{3}^{(k)})$.

Since $q>2, 1\leq p\leq q-2\leq l$, we have $\frac{g+2}{2}\leq q \leq g-1$.
Therefore, when $g=p+q>3$, $\{C_{3,n,p,q,m-p-q}^{k}~|~q=2\text{~or~}\frac{g+2}{2}\leq q \leq g-1\}$ have minimum $3k$th order spectral moment in all hypergraphs with maximum degree 2 and girth $g$ in $\{\mathcal{C}_{3,n,p,q,l}^{k}~|~ q>2, 1\leq p\leq q-2\leq l\text{~or~} q=2,1\leq p\leq l  \text {~or~}q=1, k>3,1<p\leq l\} $.

When $g=4$, we consider the following 2 cases.

Case 1. $q=2$.  For $m=6$, $N_{C_{3,n,2,2,2}^{k}}(P_{4}^{(k)})=6$, $N_{C_{3,n,2,2,2}^{k}}(C_{4}^{(k)})=2$, $N_{C_{3,n,2,2,2}^{k}}(W_{4})=0$ and $N_{C_{3,n,2,2,2}^{k}}(Q_{4})=2.$ For $m>6$, $N_{C_{3,n,2,2,m-4}^{k}}(P_{4}^{(k)})=m-4+2+6=m+4$, $N_{C_{3,n,2,2,m-4}^{k}}(C_{4}^{(k)})=1$, $N_{C_{3,n,2,2,m-4}^{k}}(W_{4})=0$ and $N_{C_{3,n,2,2,m-4}^{k}}(Q_{4})=2.$

Case 2. $q=3$. For $m=5$, $N_{C_{3,n,1,3,1}^{k}}(P_{4}^{(k)})=0$, $N_{C_{3,n,1,3,1}^{k}}(C_{4}^{(k)})=2$, $N_{C_{3,n,1,3,1}^{k}}(W_{4})=0$ and $N_{C_{3,n,1,3,1}^{k}}(Q_{4})=2.$
For $m>5$, $N_{C_{3,n,1,3,m-4}^{k}}(P_{4}^{(k)})=m-4+3+4=m+3$, $N_{C_{3,n,1,3,m-4}^{k}}(C_{4}^{(k)})=1$, $N_{C_{3,n,1,3,m-4}^{k}}(W_{4})=0$ and $N_{C_{3,n,1,3,m-4}^{k}}(Q_{4})=2.$


Hence, when $g=4,m=5$, $C_{3,n,1,3,1}^{k}$ is the first hypergraph in an $S$-order of all  hypergraphs with girth $4$ in $\mathcal{C}_{n}^{k}$.
When $g=4,m=6$, we have
\begin{align*}
&\mathrm{S}_{4k}(C_{3,n,2,2,2}^{k})-\mathrm{S}_{4k}(C_{3,n,1,3,2}^{k})\\
&=4k(k-1)^{n}(8\frac{k^{4k-8}}{8(k-1)^{4k-3}}(N_{C_{3,n,2,2,2}^{k}}(P_{4}^{(k)})-N_{C_{3,n,1,3,2}^{k}}(P_{4}^{(k)}))\\
&+(16\frac{8k^{4k-9}}{16(k-1)^{4k-4}}+2\frac{16k^{4k-9}}{16(k-1)^{4k-4}})(N_{C_{3,n,2,2,2}^{k}}(C_{4}^{(k)})-N_{C_{3,n,1,3,2}^{k}}(C_{4}^{(k)})))\\
&=4k^{4k-8}(k-1)^{n-4k+3}(7k-10)>0.
\end{align*}
Hence, when $m\geq6$, 
$C_{3,n,1,3,m-4}^{k}$ is the first hypergraph in an $S$-order of all  hypergraphs with girth $4$ in $\mathcal{C}_{n}^{k}$.

When $g>4$, we consider the following 2 cases.

Case 1. $q=2$.  We have $N_{C_{3,n,p,2,m-p-2}^{k}}(P_{4}^{(k)})=p+2+m-2-p+2+6=m+8$ and $N_{C_{3,n,p,2,m-p-2}^{k}}(Q_{4})=2.$


Case 2. $q\geq4$. We get $4<g\leq 2q-2$, $N_{C_{3,n,p,q,m-p-q}^{k}}(P_{4}^{(k)})=p+q+m-q-p+q-(q-3)+4=m+7$ and $N_{C_{3,n,p,q,m-p-q}^{k}}(Q_{4})=2.$

Hence, when $m\geq7$, $C_{3,n,1,4,m-5}^{k}$ is the first hypergraph in an $S$-order of all  hypergraphs with girth $5$ in $\mathcal{C}_{n}^{k}$.

When $g>5$, $\{C_{3,n,g-q,q,m-g}^{k}~|~\frac{g+2}{2}\leq q \leq g-1\}$ have minimum $4k$th order spectral moment in $\{C_{3,n,g-q,q,m-g}^{k}~|~q=2\text{~or~}\frac{g+2}{2}\leq q \leq g-1\}$.

Let $i\geq 5$.
For $g\geq 2i-4,m-g\geq i-2$, we consider the $ik$th order spectral moments of hypergraphs  in $\{C_{3,n,g-q,q,m-g}^{k}~|~\frac{g+2}{2}\leq q \leq g-i+4\}$.

Since $g\geq 2i-4$ and $i\geq 5$, the subhypergraph with $i$ edges which contain the hypercycle does not exist. There are only 2 edges $f_{1}, f_{q}$ of hypergraphs in $\{C_{3,n,g-q,q,m-g}^{k}~|~\frac{g+2}{2}\leq q \leq g-i+4\}$ that contain 3 vertices whose degree is equal to 2, and other edges contain at most 2 vertices whose degree is equal to 2. The minimum value of the distance between any vertex in $f_{1}$ and any vertex in $f_{q}$ is $i-4$. If there are 2 edges $e, e'$ of hypertrees with $i$ edges that contain 3 vertices of degree 2. The maximum value of the minimum value of distance between any vertex in $e$ and any vertex in $e'$ is $i-6$. Hence, there is 0 or 1 edge of subhypertrees with $i$ edges that contain 3 vertices whose degree is equal to 2, and other edges contain at most 2 vertices whose degree is equal to 2.

If there is 1 edge $e$ of subhypertrees with $i$ edges that contain 3 vertices whose degree is equal to 2, then the longest length of three pendant edges with different terminal vertices in $e$ is $i-3$.
Since $g\geq 2i-4, i\geq 5$ and $q\geq \frac{g+2}{2}$, $q\geq 4$.
Since $g-q\geq i-4, q\geq \frac{g+2}{2}\geq i-1$ and $m-g\geq i-2$, the number of subhypertrees with $i$ edges except $P_{i}^{(k)}$ of hypergraphs in $\{C_{3,n,g-q,q,m-g}^{k}~|~\frac{g+2}{2}\leq q \leq g-i+4\}$ is equal.
We have $N_{C_{3,n,i-4,g-i+4,m-g}^{k}}(P_{i}^{(k)})=g+m-g+g-i+4-(g-i+4)+i-1+2(i-2)+1=m+3i-4$ and $N_{C_{3,n,g-q,q,m-g}^{k}}(P_{i}^{(k)})=g+m-g+q-q+i-1+2(i-2)=m+3i-5,\frac{g+2}{2}\leq q <g-i+4.$ Hence, $\{C_{3,n,g-q,q,m-g}^{k}~|~\frac{g+2}{2}\leq q <g-i+4\}$ have minimum $ik$th order spectral moment  in $\{C_{3,n,g-q,q,m-g}^{k}~|~\frac{g+2}{2}\leq q \leq g-i+4\}$.

Hence, when $g=2i-4$, $C_{3,n,g-\frac{g+2}{2},\frac{g+2}{2},m-g}^{k}$ is the first hypergraph in an $S$-order of all  hypergraphs with girth $g$ in $\mathcal{C}_{n}^{k}$. When $g=2i-3$, $C_{3,n,g-\lceil\frac{g+2}{2}\rceil,\lceil\frac{g+2}{2}\rceil,m-g}^{k}$ is the first hypergraph in an $S$-order of all  hypergraphs with girth $g$ in $\mathcal{C}_{n}^{k}$.
\end{proof}
The following Theorem gives the first hypergraph in an $S$-order of all linear bicyclic $k$-uniform hypergraphs with $n$ vertices, $m$ edges and girth $g$.
\begin{thm}
In an $S$-order of all linear bicyclic $k$-uniform hypergraphs with $n$ vertices, $m$ edges and girth $g$.
\begin{equation*}
\text{the first hypergraph is }\begin{cases}
B_{3,n,g,t-4,t+1}^{k},& \text{if~} m=2t+g-3~(t>g+1),\notag \\
B_{3,n,g,t-3,t+1}^{k},& \text{if~} m=2t+g-2~(t\geq g+1).
\end{cases}
\end{equation*}
\end{thm}
\begin{proof}
For $m-3-q>0$ and $q>4$, $N_{B_{3,n,3,m-3-q,q}^{k}}(P_{3}^{(k)})=m+1$, $N_{B_{3,n,3,m-3-q,q}^{k}}(C_{3}^{(k)})=1$. And $N_{C_{3,n,1,2,m-3}^{k}}(P_{3}^{(k)})=m+1$, $N_{C_{3,n,1,2,m-3}^{k}}(C_{3}^{(k)})=1$. Therefore, the $3k$th order spectral moments of $\{B_{3,n,3,m-3-q,q}^{k}~|~m-3-q>0,q>4\}$ and $C_{3,n,1,2,m-3}^{k}$ are the equal.
When $m>8$, $N_{B_{3,n,3,m-3-q,q}^{k}}(P_{4}^{(k)})=m+2$, $N_{B_{3,n,3,m-3-q,q}^{k}}(W_{4})=1$, $N_{B_{3,n,3,m-3-q,q}^{k}}(C_{4}^{(k)})=0$, $N_{B_{3,n,3,m-3-q,q}^{k}}(Q_{4})=1,$ $N_{C_{3,n,1,2,m-3}^{k}}(P_{4}^{(k)})=m+1$, $N_{C_{3,n,1,2,m-3}^{k}}(W_{4})=2$, $N_{C_{3,n,1,2,m-3}^{k}}(C_{4}^{(k)})=0$ and $N_{C_{3,n,1,2,m-3}^{k}}(Q_{4})=0.$
Then by Equation (\ref{fan111}), we have
\begin{align*}
&\mathrm{S}_{4k}(C_{3,n,1,2,m-3}^{k})-\mathrm{S}_{4k}(B_{3,n,3,m-3-q,q}^{k})\\
&=4k(k-1)^{n}(8\frac{k^{4k-8}}{8(k-1)^{4k-3}}(N_{C_{3,n,1,2,m-3}^{k}}(P_{4}^{(k)})-N_{B_{3,n,3,m-3-q,q}^{k}}(P_{4}^{(k)}))\\
&+8\frac{k^{4k-8}}{8(k-1)^{4k-3}}(N_{C_{3,n,1,2,m-3}^{k}}(Q_{4})-N_{B_{3,n,3,m-3-q,q}^{k}}(Q_{4}))\\
&+(16\frac{6k^{4k-9}}{16(k-1)^{4k-4}}+4\frac{8k^{4k-9}}{16(k-1)^{4k-4}})(N_{C_{3,n,1,2,m-3}^{k}}(W_{4})-N_{B_{3,n,3,m-3-q,q}^{k}}(W_{4})))\\
&=8k^{4k-8}(k-1)^{n-4k+3}(3k-4)>0.
\end{align*}
Hence, in an $S$-order of all linear bicyclic $k$-uniform hypergraphs with $n$ vertices, $m$ edges and girth $3$, the first hypergraph is in $\{B_{3,n,3,m-3-q,q}^{k}~|~ m-3-q>0,q>4\}$.

When $g\geq 4$. For $q\geq g$ and $m-g-q\geq 0$, we have $N_{B_{3,n,g,m-g-q,q}^{k}}(P_{3}^{(k)})=m+4, N_{B_{3,n,g,m-g-q,q}^{k}}(C_{3}^{(k)})=0$. For $q=2\text{~or~}\frac{g+2}{2}\leq q \leq g-1$, we have $N_{C_{3,n,p,q,m-p-q}^{k}}(P_{3}^{(k)})=m+4, N_{C_{3,n,p,q,m-p-q}^{k}}(C_{3}^{(k)})=0$.
Hence, for $g\geq 4$, the $3k$th order spectral moments of $\{B_{3,n,g,m-g-q,q}^{k}~|~q\geq g,m-g-q\geq 0\}$ and $\{C_{3,n,p,q,m-p-q}^{k}~|~q=2\text{~or~}\frac{g+2}{2}\leq q \leq g-1\}$ are the equal.

When $m>11$. For $m-4-q>1$ and $q>5$, $N_{B_{3,n,4,m-4-q,q}^{k}}(P_{4}^{(k)})=m+3$, $N_{B_{3,n,4,m-4-q,q}^{k}}(C_{4}^{(k)})=1$, $N_{B_{3,n,4,m-4-q,q}^{k}}(W_{4})=0$, $N_{B_{3,n,4,m-4-q,q}^{k}}(Q_{4})=2$. And $N_{C_{3,n,1,3,m-4}^{k}}(P_{4}^{(k)})=m+3$, $N_{C_{3,n,1,3,m-4}^{k}}(C_{4}^{(k)})=1$, $N_{C_{3,n,1,3,m-4}^{k}}(W_{4})=0$, $N_{C_{3,n,1,3,m-4}^{k}}(Q_{4})=2.$
Therefore, the $4k$th order spectral moments of $\{B_{3,n,4,m-4-q,q}^{k}~|~m-4-q>1,q>5\}$ and $C_{3,n,1,3,m-4}^{k}$ are the equal.
We have $N_{C_{3,n,1,3,m-4}^{k}}(P_{5}^{(k)})=m-4+3+5=m+4, N_{C_{3,n,1,3,m-4}^{k}}(Q_{5})=2$, $N_{C_{3,n,1,3,m-4}^{k}}(W_{5})=2$, $N_{B_{3,n,4,m-4-q,q}^{k}}(P_{5}^{(k)})=m+4, N_{B_{3,n,4,m-4-q,q}^{k}}(Q_{5})=4$, $N_{B_{3,n,4,m-4-q,q}^{k}}(W_{5})=1$. Then by Equation (\ref{fan111}), we have
\begin{align*}
&\mathrm{S}_{5k}(C_{3,n,1,3,m-4}^{k})-\mathrm{S}_{5k}(B_{3,n,4,m-4-q,q}^{k})\\
&=5k(k-1)^{n}(16\frac{k^{5k-10}}{16(k-1)^{5k-4}}(N_{C_{3,n,1,3,m-4}^{k}}(Q_{5})-N_{B_{3,n,4,m-4-q,q}^{k}}(Q_{5}))\\
&+(32\frac{8k^{5k-11}}{32(k-1)^{5k-5}}+4\frac{16k^{5k-11}}{32(k-1)^{5k-5}})(N_{C_{3,n,1,3,m-4}^{k}}(W_{5})-N_{B_{3,n,4,m-4-q,q}^{k}}(W_{5})))\\
&=10k^{5k-10}(k-1)^{n-5k+4}(4k-5)>0.
\end{align*}
Hence, in an $S$-order of all linear bicyclic $k$-uniform hypergraphs with $n$ vertices, $m$ edges and girth $4$, the first hypergraph is in $\{B_{3,n,4,m-4-q,q}^{k}~|~ m-4-q>1,q>5\}$.

When $g\geq 5$. For $m-g-q>0, q\geq g$, we have $N_{B_{3,n,g,m-g-q,q}^{k}}(P_{4}^{(k)})=m+7$, $N_{B_{3,n,g,m-g-q,q}^{k}}(W_{4})=0$ and $N_{B_{3,n,g,m-g-q,q}^{k}}(Q_{4})=2$. For $\frac{g+2}{2}\leq q \leq g-1$, we have $N_{C_{3,n,p,q,m-p-q}^{k}}(P_{4}^{(k)})=m+7, N_{C_{3,n,p,q,m-p-q}^{k}}(W_{4})=0$ and $N_{C_{3,n,p,q,m-p-q}^{k}}(Q_{4})=2.$
Hence, for $g\geq 5$, the $4k$th order spectral moments of $\{B_{3,n,g,m-g-q,q}^{k}~|~m-g-q>0,q\geq g\}$ and $\{C_{3,n,p,q,m-p-q}^{k}~|~\frac{g+2}{2}\leq q \leq g-1\}$ are the equal.

When $m>12$.
For $C_{3,n,1,4,m-5}^{k}$ and $\{B_{3,n,5,m-5-q,q}^{k}~|~ m-5-q>1,q>5\}$, since the number of the other subhypergraphs with $5$ edges except $P_{5}^{(k)}$ is equal, we only need consider the number of $P_{5}^{(k)}$.
For $m-5-q>1$ and $q>5$, we have  $N_{B_{3,n,5,m-5-q,q}^{k}}(P_{5}^{(k)})=m+5$.
And $N_{C_{3,n,1,4,m-5}^{k}}(P_{5}^{(k)})=m+6$. Therefore, for $m-5-q>1$ and $q>5$, we have $\mathrm{S}_{5k}(C_{3,n,1,4,m-5}^{k})>\mathrm{S}_{5k}(B_{3,n,5,m-5-q,q}^{k}).$
Hence, in an $S$-order of all linear bicyclic $k$-uniform hypergraphs with $n$ vertices, $m$ edges and girth $5$, the first hypergraph is in $\{B_{3,n,5,m-5-q,q}^{k}~|~ m-5-q>1,q>5\}$.

When $g>5$ and $t\geq 5$. If $t<g$, for $m-g-q>t-4$ and $q\geq g$, we have $N_{B_{3,n,g,m-g-q,q}^{k}}(P_{t}^{(k)})=m+3t-5$.
If $t=g$ and $m>3g-3$, for $m-g-q>g-4$ and $q>g$, we have $N_{B_{3,n,g,m-g-q,q}^{k}}(P_{g}^{(k)})=2g-5+m$ and $N_{B_{3,n,g,m-g-q,q}^{k}}(C_{g}^{(k)})=1$.
If $t>g$ and $m>2t+g-3$, for $m-g-q>t-4$ and $q>t$, we have $N_{B_{3,n,g,m-g-q,q}^{k}}(P_{t}^{(k)})=t+m+g-5$ and $N_{B_{3,n,g,m-g-q,q}^{k}}(C_{t}^{(k)})=0$.

Let $t\geq 5$. For $g\geq 2t-4$ and $m-g\geq t-2$,  when $g$ is even, if $t\leq\frac{g}{2}+2$, we have $N_{C_{3,n,g-\frac{g+2}{2},\frac{g+2}{2},m-g}^{k}}(P_{t}^{(k)})=m+3t-5$.
If $\frac{g}{2}+2<t<g$, we have $N_{C_{3,n,g-\frac{g+2}{2},\frac{g+2}{2},m-g}^{k}}(P_{t}^{(k)})=g+m-g+\frac{g+2}{2}+2(t-2)+t-3-g+\frac{g+2}{2}=m+3t-5$.
If $t=g$, we have $N_{C_{3,n,g-\frac{g+2}{2},\frac{g+2}{2},m-g}^{k}}(P_{g}^{(k)})=m-g+\frac{g+2}{2}+2(g-2)+g-3-g+\frac{g+2}{2}=m+2g-5$.
If $5\leq t\leq \frac{g}{2}+3$, for $C_{3,n,g-\frac{g+2}{2},\frac{g+2}{2},m-g}^{k}$ and $\{B_{3,n,g,m-g-q,q}^{k}~|~ m-g-q>t-4,q>g\}$, the number of the subhypergraphs with $t$ edges is equal.

When $g$ is even, $g\geq 8$ and $t=\frac{g}{2}+4$, for $C_{3,n,g-\frac{g+2}{2},\frac{g+2}{2},m-g}^{k}$ and $\{B_{3,n,g,m-g-q,q}^{k}~|~ m-g-q>\frac{g}{2},q>g\}$, the number of the subhypergraphs with $\frac{g}{2}+4$ edges except $Q_{\frac{g}{2}+4}$ is equal. And for $m-g-q>\frac{g}{2}$ and $q>g$, we have
$N_{B_{3,n,g,m-g-q,q}^{k}}(Q_{\frac{g}{2}+4})<N_{C_{3,n,g-\frac{g+2}{2},\frac{g+2}{2},m-g}^{k}}(Q_{\frac{g}{2}+4}).$ Hence, when $g$ is even and $g\geq 8$, in an $S$-order of all linear bicyclic $k$-uniform hypergraphs with $n$ vertices, $m$ edges and girth $g$, the first hypergraph is in $\{B_{3,n,g,m-g-q,q}^{k}~|~ m-g-q>\frac{g}{2},q>g\}$.

When $g=6$.  For $t=\frac{g}{2}+4=7$, we have $N_{C_{3,n,2,4,m-6}^{k}}(P_{7}^{(k)})=m-6+4+8+2=m+8$ and
$N_{C_{3,n,2,4,m-6}^{k}}(W_{7})=2$. For $m-6-q>3$ and $q>7$, we have $N_{B_{3,n,6,m-6-q,q}^{k}}(P_{7}^{(k)})=m+8$ and $N_{B_{3,n,6,m-6-q,q}^{k}}(W_{7})=1$.
The number of subhypertrees with $7$ edges of $B_{3,n,6,m-6-q,q}^{k}$ is less than or equal to the number of subhypertrees with $7$ edges of $C_{3,n,2,4,m-6}^{k}.$
Hence,
in an $S$-order of all linear bicyclic $k$-uniform hypergraphs with $n$ vertices, $m$ edges and girth $6$, the first hypergraph is in $\{B_{3,n,6,m-6-q,q}^{k}~|~ m-6-q>3,q>7\}$.

When $g\geq 7$ and $g$ is odd, if $5\leq t\leq\frac{g}{2}+\frac{3}{2}$, we have $N_{C_{3,n,g-\lceil\frac{g+2}{2}\rceil,\lceil\frac{g+2}{2}\rceil,m-g}^{k}}(P_{t}^{(k)})=m+3t-5$.
For $C_{3,n,g-\lceil\frac{g+2}{2}\rceil,\lceil\frac{g+2}{2}\rceil,m-g}^{k}$ and $\{B_{3,n,g,m-g-q,q}^{k}~|~ m-g-q>t-4,q>g\}$, the number of the subhypergraphs with $t$ edges is equal. If $t=\frac{g}{2}+\frac{5}{2}$, we have $N_{C_{3,n,g-\lceil\frac{g+2}{2}\rceil,\lceil\frac{g+2}{2}\rceil,m-g}^{k}}(P_{\frac{g}{2}+\frac{5}{2}}^{(k)})=g+m-g+\lceil\frac{g+2}{2}\rceil+2(\frac{g}{2}+\frac{5}{2}-2)+1=m+\frac{3g}{2}+\frac{7}{2}$.
For $m-g-q>\frac{g}{2}-\frac{3}{2}$ and $q>g$, we have $N_{B_{3,n,g,m-g-q,q}^{k}}(P_{\frac{g}{2}+\frac{5}{2}}^{(k)})=m+\frac{3g}{2}+\frac{5}{2}$.
For $C_{3,n,g-\lceil\frac{g+2}{2}\rceil,\lceil\frac{g+2}{2}\rceil,m-g}^{k}$ and $\{B_{3,n,g,m-g-q,q}^{k}~|~ m-g-q>\frac{g}{2}-\frac{3}{2},q>g\}$, the number of the subhypergraphs with $\frac{g}{2}+\frac{5}{2}$ edges except $P_{\frac{g}{2}+\frac{5}{2}}^{(k)}$ is equal.
Hence, when $g\geq 7$ and $g$ is odd,
in an $S$-order of all linear bicyclic $k$-uniform hypergraphs with $n$ vertices, $m$ edges and girth $g$, the first hypergraph is in $\{B_{3,n,g,m-g-q,q}^{k}~|~ m-g-q>\frac{g}{2}-\frac{3}{2},q>g\}$.


\end{proof}
%
%
\vspace{3mm}

\noindent
\textbf{Acknowledgments}
\vspace{3mm}
\noindent

This work is supported by the National Natural Science Foundation of China (No. 12071097, No. 12371344), the Natural Science Foundation for The Excellent Youth Scholars of the Heilongjiang Province (No. YQ2022A002) and the Fundamental Research Funds for the
Central Universities.

\section*{References}
\bibliographystyle{unsrt}
\bibliography{pbib}
\end{spacing}
\end{document}